\documentclass[amstex,12pt,reqno]{amsart}

\usepackage{libertinus}            

\usepackage{makecell}

\usepackage{amsmath,amsthm,amscd,amsfonts,amssymb,enumitem}
\usepackage{latexsym}
\usepackage{multirow}
\usepackage{mathrsfs}
\usepackage{xcolor}
\usepackage{lipsum} 
\usepackage{graphicx}
\usepackage{multirow}
\usepackage[margin=1in]{geometry}
\usepackage{subcaption}
\usepackage{caption}
\usepackage{algorithm}
\usepackage{algpseudocode}
\usepackage{geometry}
\usepackage{subcaption}
\usepackage{amsmath} 
\usepackage{array}   
\usepackage{booktabs} 
\usepackage{colortbl} 
\usepackage[table,xcdraw]{xcolor}
\usepackage{algorithm}
\usepackage{algpseudocode}

\usepackage{gettitlestring}[2009/12/18]
\usepackage{kvoptions}
\usepackage{lscape} 
\newtheorem{thm}{Theorem}[section]
\newtheorem{lemma}[thm]{Lemma}
\newtheorem{cor}[thm]{Corollary}

\theoremstyle{definition}
\newtheorem{defin}[thm]{Definition}

\newtheorem{rem}[thm]{Remark}

\newcommand{\qedwhite}{\hfill \ensuremath{\Box}}
\renewenvironment{proof}{{\raggedright \bfseries Proof.}}{\qedwhite}

\numberwithin{equation}{section}

\renewcommand{\arraystretch}{1.5}

\def\NN{\mathbb N}
\def\ZZ{\mathbb Z}
\def\RR{\mathbb R}

\def\Chi{\mbox{\large$\chi$}}
\begin{document}
	
	\leftline{ \scriptsize \it  }
	\title[]
{Neural Network Approximation in mixed norm Spaces with applications in Image Processing}
	\maketitle
	
	\begin{center}
		{\bf Priyanka Majethiya,}
		{\bf Shivam Bajpeyi\footnote{corresponding author\\
				Email: shivambajpai1010@gmail.com, shivambajpeyi@amhd.svnit.ac.in }}
		
		\vskip0.2in
		Department of Mathematics, Sardar Vallabhbhai National Institute of Technology Surat, Gujarat-395007\\
		
		\verb" priyankamajethiya2000@gmail.com, shivambajpai1010@gmail.com" 
	\end{center}
	
	\begin{abstract}
		On the one hand, the framework of mixed norm spaces has potential applications in different areas of mathematics. On the other hand,
		neural network (NN) operators are well established as approximators, attracting significant attention in the fields of approximation theory and signal analysis.
		In this article, we aim to integrate both these concepts. 
		Here, we analyze multivariate Kantorovich-type NN operators as approximators based on sigmoidal activation function within various mixed norm structures, namely mixed norm Lebesgue spaces $L^{\overrightarrow{\mathscr{P}}}([-1,1]^n),$ where $\overrightarrow{\mathscr{P}}=(p_1,...,p_n),$ and mixed norm Orlicz spaces $L^{\overrightarrow{\Phi}}([-1,1]^n)$, where $\overrightarrow{\Phi}=(\phi_1,...,\phi_n)$. The Orlicz spaces consist of various significant function spaces, such as  classical \textit{Lebesgue spaces,  Zygmund spaces and Exponential  spaces}. We establish the boundedness and convergence of multivariate Kantorovich-type NN operators in these mixed norm function spaces. 
		At the end, the approximation abilities are illustrated through graphical representations and error-estimates along with
		application in image processing, including image reconstruction, image scaling, image inpainting and image denoising. Finally, we address the beneficial impact of mixed norm structure on aforementioned image processing tasks.\\
		
		

		\vskip0.001in
		
		
	\end{abstract}
	\section{introduction} \label{1}
	
	Sampling and reconstruction are essential processes that convert continuous functions into discrete sequences, thereby effectively bridging the gap between analog and digital representations. In this direction,
the Whittaker–Kotel’nikov–Shannon (WKS) theorem \cite{noise} is a foundational result concerning the reconstruction of band-limited functions from their sampled values  on a uniform grid. Specifically, it states that if a function \( f : \mathbb{R} \rightarrow \mathbb{R} \) is square integrable, i.e., \( f \in L^2(\mathbb{R}) \), and band-limited with its Fourier transform 
	vanishing outside \( [-\pi \omega, \pi \omega] \), then $f$ can be completely reconstructed by the formula
	\[
	f(x) = \sum_{k \in \mathbb{Z}} f\left(\frac{k}{\omega}\right) \, \emph{sinc} \left(  \omega x - k \right),  
	\]
	where $\emph{sinc}(x) :=\frac{\sin \pi x} {\pi x}$ for $x\neq 0$ and $\emph{sinc}(0)=1$. Several significant developments related to the Whittaker-Kotelnikov-Shannon sampling theorem and its various generalizations are documented in \cite{jerri}. The Shannon Sampling Theorem has made a significant contribution to the field of function approximation and has remained a foundational tool in signal analysis and reconstruction over the years \cite{siam}.
	\par
	
	Nowadays, Artificial Neural Networks (ANNs) have gained significant attention for their outstanding proficiency in function approximation. Among various architectures of ANN, the feedforward neural network (FNN) is one of the most commonly studied.
	A feedforward neural network (FNN) consisting of a set of processing elements known as neurons, is recognized as one of the most important architectures in ANN and is identified as a \textit{universal approximator}. 
	
	\begin{thm} \textit{(Universal Approximation Theorem, \cite{pinkus}).} 
		Let $n \geq 1$ be a dimension and $\rho : \mathbb{R} \rightarrow \mathbb{R}$ continuous. Then the following are equivalent:
		
		\begin{itemize}
			\item[(a)] For any compact $A \subset \mathbb{R}^n$, any continuous $f : A \rightarrow \mathbb{R}$ and any $\delta > 0$, there exists a network with one hidden layer and activation function $\rho $, representing $g : \mathbb{R}^n \rightarrow \mathbb{R}$, such that for all $x \in A$,
			\[
			|f(x) - g(x)| \leq \delta.
			\]
			
			\item[(b)] $\rho $ is not a polynomial.
		\end{itemize}
	\end{thm}
	The efficiency of neural networks in approximating or learning any functions  has made significant advancements in recent years in various domains such as approximation theory, signal analysis and image processing  \cite{applkadak,kadakapl,kadakfuzzy, imgpn}.

	Neural network (NN) operators developed  from the mathematical formulation of FNNs, provide a constructive method to approximate target functions utilizing suitable activation functions.   
	One significant advantage of using such operators is that they are well-recognized as effective approximators, as all components, including weights, thresholds, and coefficients, of a FNN can be explicitly determined.
	The mathematical formulation of FNNs can be expressed as follows. Consider  \( \rho : \mathbb{R} \rightarrow \mathbb{R} \) as an activation function, \( n \) inputs, one output, and one hidden layer with \( m \) neurons can be mathematically expressed as
	\begin{equation}
		K_{n,\rho}(\mathbf{x}) = \sum_{i=1}^{m} c_i \hspace{1pt}\rho(\mathbf{w}_i \cdot \mathbf{x} + b_j), \quad \mathbf{x} = (x_1, x_2, \dots, x_n) \in \mathbb{R}^n, \quad n \in \mathbb{N},
	\end{equation}
	where \( c_i \in \mathbb{R}  \) are the coefficients,  \( b_i \in \mathbb{R} \) are the thresholds and \( \mathbf{w}_i = (w_{i1}, w_{i2}, \dots, w_{in}) \in \mathbb{R}^n \) are the weights and \( \left(\mathbf{w}_i \cdot \mathbf{x}\right) \) denotes the inner product of $\mathbf{w}_i$ and $\mathbf{x}.$   
	\par 
	In recent years, numerous significant results have been established on various topics in FNNs, including density results \cite{Agrawal, BA1, spigler,NN}, quantitative estimations that describe the approximation order \cite{Anas1,Anas2,Anas3, sainz}, and the complexity of approximation, which characterizes the internal complexity of neural networks \cite{spigler, barron, Funahashi, pinkus}.
	The foundational work of Cybenko \cite{cybenko} and Funahasi \cite{Funahashi} initiated the mathematical framework for function approximation by FNNs and derived the approximation results for continuous as well as measurable functions. The proof of Cybenko result used the tool of Riesz Representation theorem and Hahn-Banach theorem. Building on Cybenko's framework,
	Hornik \cite{Hornik} employed  Fourier analysis in groups and Sobolev space to develop $L^p(\mu)$ and derivative approximation theorems. 
	Moreover, Leshno et al. \cite{Leshno} demonstrated that FNNs exhibit universality for any non-polynomial activation function. The density of FNNs in the space of continuous functions is established by a measure-theoretic approach in \cite{mes}. A topological neural network (TNN) is developed in \cite{top} based on a strong universal approximation theorem for Tychonoff spaces.  Noteworthy developments in this direction can be found in \cite{pinkus,Anas3,Anas2,2, NN}. 
	\par 
	These theoretical results on universal approximation are limited to proving the existence of FNNs.
	  The foundational work of Cardaliaguet and Euvrard \cite{CE} introduced a constructive approach that demonstrates sequences of bell-shaped and squishing operators (\ref{e1}) through the approximations of continuous and bounded functions is defined as follows
	\begin{equation}\label{e1}
		f_n(x) = \displaystyle\sum_{k=-n^2}^{n^2} \frac{f\left(\frac{k}{n}\right)}{I \cdot n^\alpha} \hspace{5pt}
		b \left( n^{1-\alpha} \left( x - \frac{k}{n} \right) \right) ,
	\end{equation}
	where 
	\begin{equation*}
		I = \int_{-\infty}^{+\infty} b(t) \, dt, \quad 0 < \alpha < 1.
	\end{equation*}
	Building on this foundation, various neural network (NN) operators have been developed.
	Cao and Zang \cite{cao11} constructed a Feedforward neural networks (FNNs) and derived bounds for its approximation error in the $L^p-$setting. Llanas and Sainz \cite{sainz} examined FNNs with interpolation and approximation properties. 
	Moreover, in \cite{2}  Costarelli introduced a neural network interpolation operator and employing a wide range of smooth univariate activation function within the neural network operator.
Kadak contributed substantially to the development of NN operators and their applications in areas such as image processing and signal analysis \cite{applkadak,kadakapl,kadakfuzzy}.  Qian and Yu \cite{qyu} investigated the approximation properties  of NN operators for particular activation functions. Recently, Yu and Cao \cite {Yu} constructed suitable FNNs utilizing sigmoidal function satisfying certain asymptotic behaviours. 	In recent years, there has been considerable interest in constructing effective NN operators, see \cite{cao11,Anas2,WYG,BA1, Agrawal, q1, cao2} for details.

  Further, Costarelli and Spigler \cite{NN} pioneered the study of bounded functions through a family of neural network (NN) operators induced by a general class of activation functions. Subsequently, the same authors extended this study by replacing the sample values with an average of the signal \( f \) over an interval containing \( k/n \),  leading to a Kantorovich-type variant of the operator. These operators are of considerable interest, as from a practical standpoint, information is often better captured in a small region around a point rather than at the point itself, thereby  reducing jitter-errors.
	 Costarelli et al.\cite{spigler} studied the convergence of a family of multivariate Kantorovich-type NN operators, which is defined as
	\begin{equation}\label{intr}
		K_n(f,\textbf{x})= \frac{\displaystyle\sum_{\textbf{k}=-n}^{n-1}  \left(n^r\int_{I_{\textbf{k},n}}f(\textbf{t})d\textbf{t}\right)\phi_{\rho}(n\textbf{x}-\textbf{k})}{\displaystyle\sum_{\textbf{k}=-n}^{n-1} \phi_{\rho}(n\textbf{x}-\textbf{k})}.
	\end{equation}
	\par
	
	Building on this development, we analyze the approximation abilities of \textit{multivariate Kantorovich-type NN operators} based on the mathematical formulation of feedforward neural networks. We will examine the approximation properties of these operators within the framework of mixed norm function spaces, specifically mixed norm Lebesgue and Orlicz spaces. Along with this, we explore the applications of these operators in image processing tasks such as image reconstruction, image inpainting, image denoising, image scaling. Subsequently, we discuss the application of mixed norm structures in both signal and image processing contexts.
	
	\par
	The Orlicz spaces provide a general framework that encompass classical $L^p$-spaces and various other significant spaces, such as \textit{Exponential spaces} and \textit{Zygmund spaces} (see Section \ref{2}).
	The theory of Orlicz space was established by W. Orlicz in \cite{orl}. 
	The  rigorous theoretical foundation of Orlicz spaces was introduced by Rao and Ren in \cite{Raobook,13}. The class of function spaces constructed from Orlicz spaces has shown considerable applicability across several mathematical domains, notably in functional analysis and operator theory \cite{log2,exp}. For a through detailed analysis on fundamental applications of Orlicz spaces in various areas of mathematics, see \cite{orliczapp, Raobook}. Moreover, some significant advancements of Orlicz spaces in the field of approximation theory are detailed in \cite{k2007,nn,AAmine, vayeda}. Hence, Orlicz spaces provide a unified approach to study the multivariate Kantorovich NN operators over various functional settings. Specifically, we focus on mixed norm function spaces due to their general framework and diverse areas of applications.
	\par 
	
	\par


	It is noteworthy that the study in mixed norm structure is fundamentally
	different from a usual multivariate generalization. 
	While usual multi-dimensional generalizations typically
	consider uniform norms applied across all the dimensions, mixed norm spaces introduce flexibility by
	allowing different norms in different directions or components. This anisotropic behavior can capture
	variable integrability and smoothness within the same space, which can not be addressed by traditional
	multivariate frameworks. Thus, mixed norm spaces provide a richer structure for analyzing these
	inequalities and operators, enabling results that extend beyond those achievable in usual multivariate
	generalizations. 	These function spaces are well-suited for handling time-dependent signals \cite{sun} and serve a central role in the theoretical investigation of PDEs in both time and space domains, such as the heat and wave equations \cite{heat}.
	The origins of this concept is found in the seminal
	 work of Benedek and Panzone \cite{dct}.
	An operator-theoretic approach to mixed norm spaces have been explored in \cite{ban,grey, clean, dct}.
A detailed discussion on the embedding criteria for mixed norm Lebesgue and Orlicz spaces can be found in  \cite{pq}.
	Maligranda’s work \cite{3} focused on the Calderón–Lozanovskiĭ construction, particularly its application to mixed norm function spaces generated from Banach ideal spaces. His research included a rigorous analysis of the Calderón product associated with these function spaces.
	In light of utility of mixed norm spaces and the long-standing interest for neural network operators in approximation
	theory, studying these operators in the framework of mixed norm spaces appears noteworthy.

	 This construction of neural network operators is particularly effective because it extends to a broader class of functions, including $L^p$-spaces
	and Orlicz spaces \cite{nn, spigler}.
	The study of the above family of NN operators \eqref{intr} within mixed norm spaces is not only of theoretical significance (as detailed in Section~\ref{3}-\ref{4}), but also leads to valuable impact in signal and image analysis (as shown in Subsection \ref{appmix}).

	\par
	\subsection{Contributions} 
	The key features of this article are as follows:
	\begin{itemize}
		\item The convergence of Kantorovich-type neural network (NN) operators have been studied in Lebesgue spaces \cite{spigler} and Orlicz spaces \cite{nn}. Building on this foundation, the present article offers a comprehensive analysis of the approximation properties of (\ref{intr}) within a mixed norm framework, particularly in mixed norm Lebesgue space and mixed norm  Orlicz spaces.\\
		
		\item  The mixed norm Orlicz spaces have not been extensively  explored in the existing literature. Owing to the flexible structure of Orlicz spaces, this article provides a general framework for studying a broad class of function spaces.\\
		
		
		
		\item Alongside the theoretical developments, we demonstrate applications in the field of image processing, including image reconstruction, denoising, inpainting, and scaling. Furthermore, we highlight how mixed norm structures provide significant advantages over standard usual norm setting in both signal and image processing contexts.
	\end{itemize}
	
	\subsection{Structure of the paper}	The structure of the paper is as follows.  Section \ref{2} provides the fundamental definitions and some important results which will be required for further analysis. In Section \ref{3}, first we study the boundedness as well as convergence of (\ref{multi}) within the framework of 
	$L^{\overrightarrow{\mathscr{P}}}([-1,1]^n)$. Section \ref{4} is devoted to proving these results in the context of mixed norm Orlicz spaces.
	In pursuing this study, a denisty theorem has been establish within $L^{\overrightarrow{\Phi}}([a,b]^n).$ 
	Some numerical examples have been presented with graphical representation and error estimates using specific sigmoidal functions in Section \ref{5}. In addition, we also explore applications in image processing, focusing on image reconstruction, image denoising, image inpainting, scaling through a two-dimensional Kantorovich-type NN operator. The performance of these operators has been evaluated by using the notion of Structural Similarity Index Measure (SSIM) and Peak Signal-to-Noise Ratio (PSNR). Furthermore, we discuss the advantageous effects of mixed norm structures in the contexts of signal and image processing.

		\section{Notation and Preliminary Results}\label{2}
		In the following, the notations $\mathbb{N}^r$, $\mathbb{Z}^r$, and $\mathbb{R}^r$ represent the sets of $r$-tuples  of natural numbers, integers, and real numbers respectively. Moreover, the notation $\RR^r_{+}$
			represents  the set of all \( r \)-tuples of non-negative real numbers.
		We denote $\mathcal{J} :=[-1,1]^r \left(\subset \RR^r \right)$.
		For $\textbf{x}=(x_1,...,x_r)$ and $\textbf{y}=(y_1,...,y_r)$, we denote $(\textbf{x}+\textbf{y})=(x_1+y_1,...,x_r+y_r)$, and
		$\xi\textbf{x}=(\xi x_1,...,\xi x_r)$, for any scalar $\xi \in \RR$
		. We write $\textbf{x} \leq \textbf{y}$ to mean that $x_i \leq y_i$ for each $i=1,2,...,r$ and $\textbf{x}/\textbf{y}=(x_1/y_1,...,x_r/y_r)$.
		The space of continuous function on $\mathcal{J} $, is  referred to as $C(\mathcal{J})$ equipped with the norm $\|f\|_{\infty} :=\sup_{\textbf{t} \in \mathcal{J} } |f(\textbf{t})|.$
		Furthermore, 
		$M(\mathcal{J})$  denotes the set of all measurable functions  
		on $\mathcal{J}$. Here  \( \mathcal{D}(\textbf{m}) \) is the multivariate sigmoidal functions on the box-domain \(  \RR^r\), where \(\textbf{m} \in \RR_{+}^r \).
		
		In this section, our focus is on three key concepts: sigmoidal functions, mixed norm Lebesgue spaces, and mixed norm Orlicz spaces.
		
		\subsection{Sigmoidal Functions: Definitions and Properties}\label{2.1}
		
		\begin{defin}\label{def1} (Multivariate sigmoidal function)
			A measurable function $\rho : \RR^r \rightarrow \RR$ is multivariate sigmoidal function provided $$ \lim_{\textbf{x} \rightarrow -\infty}\rho(\textbf{x})=0 \quad \text{and }\quad \lim_{ \textbf{x}\rightarrow \infty}\rho(\textbf{x})=1 .$$
			
			A multivariate sigmoidal function $ \rho(\textbf{x}) $ belongs to the class $  \mathcal{D}(\textbf{m})$ if the following conditions hold:
			\begin{itemize}
				\item[(S1)] \( \rho(\textbf{x}) \) is a non-decreasing function;
				\item[(S2)] \( \rho(\textbf{x}) = 1 \) for \( \textbf{x} \geq \mathbf{m} \), and \( \rho(\textbf{x}) = 0 \) for \( \textbf{x} \leq -\mathbf{m} \), where  \( \mathbf{m} = (m, \dots, m) \in \mathbb{R}^r_+ \).
			\end{itemize}
			
		\end{defin}
		The density function is defined using the sigmoidal function $\rho$ as follows :
		$$\Psi_{\rho}(x) := \rho(x+1)-\rho(x-1),\quad x \in \RR.$$
		
		For further analysis, we will utilize multivariate density function associated  with $\rho$, that is:
		\begin{equation} \label{we}
			\Psi_{\rho}(\textbf{x}) :=\Psi_{\rho}(x_1)\cdots  \Psi_{\rho}(x_r), \quad \textbf{x} :=(x_1,...,x_r) \in \RR^r.
		\end{equation} 
		
		The  significant properties of the multivariate density function $\Psi_\rho$ are formulated in the following lemmas.
		\begin{lemma} \cite{qyu} \label{lem1} 
			Let $\rho \in\mathcal{D}(\textbf{m}) $ and $\textbf{m} \in \RR^r_{+}$. Then we have
			\item[(i)] For every $\textbf{x} \in \RR^r$,  $\Psi_{\rho}(\textbf{x} ) \geq 0$ and $\lim_{\textbf{x}\rightarrow \pm \infty}\Psi_{\rho}(\textbf{x})=0;$ 
			\item[(ii)]  For \( \textbf{x} \in [0, 2m]^r \), \( \Psi_{\rho}(\textbf{x}) + \Psi_{\rho}(\textbf{x} - 2\mathbf{m}) = 1 \), and \( \Psi_{\rho}(\textbf{x}) + \Psi_{\rho}(\textbf{x} + 2\mathbf{m}) = 1 \) for \( \textbf{x} \in [-2m, 0]^r \);
			\item[(iii)]
			\( \Psi_{\rho}(\textbf{x}) = 0 \) for every \( \textbf{x} \in \mathbb{R}^r \) satisfying both \( x_i \geq 2m \) and \( x_i \leq -2m \), $i=1,...,r$.
		\end{lemma}	
	
	\begin{lemma}\label{lem 2.3} \cite{spigler}
		Let $\textbf{x} \in [-1,1] \times \cdots \times [-1,1]\subset \RR^r$ and $n \in \NN$. Then
		\begin{equation}\label{lem 2.2}
			\frac{1}{
				\displaystyle \prod_{i=1}^{r} \sum_{k_i=-n}^{n- 1} \Psi_\rho(n x_i - k_i)} \leq \frac{1}{\left[\Psi_\rho(2)\right]^r} .
		\end{equation}
		Here	$\lceil \cdot \rceil$  and  $\lfloor \cdot \rfloor$  represent the ceiling and floor functions, respectively.
	\end{lemma}
	
	\begin{rem} \cite{NN} \label{NN}
		(1)	 From the properties described above, it can be observe that 
		\(\rho'(0) \geq \rho'(x) \geq  0, \, x \in \mathbb{R}.\) Hence, \(\rho\) as well as \(\Psi_{\rho}\) are Lipschitz continuous on \(\mathbb{R}.\) Moreover, \(\Psi_{\rho}(x) = \mathcal{O}(|x|^{-\eta})\) as \(x \to \pm \infty\), with \(\eta\) being a positive constant. Consequently, it follows that \(\Psi_{\rho} \in L^1(\mathbb{R})\) for \(\eta > 1,\) and
		\begin{equation}\label{ce}
			\displaystyle \int_{\mathbb{R}} \Psi_\rho (t) \, dt = 1.
		\end{equation}
		(2) For all $n \in \mathbb{N}$, we have
		$$1 \geq \displaystyle\sum_{k=-n}^{n} \Psi_\rho (nx - k) \geq \Psi_\rho (1) > 0, \quad  x \in [-1,1].$$
		(3) For all $n \in \NN$ satisfies, such that
		$$\displaystyle\sum_{k= -n}^{n-1} \Psi_\rho (nx - k)\geq \Psi_\rho(2)>0, \quad x \in [-1,1].$$
	\end{rem}
	\begin{rem}\label{rem 2.4}
		(a)	The series $\displaystyle
		\sum_{\textbf{k} \in \mathbb{Z}^r} \Psi_{\rho}(\textbf{x} - \textbf{k})$
		converges uniformly on compact subsets of \(\mathbb{R}^r.\)\\
		(b)	 For every \(\textbf{x} \in \mathbb{R}^r\), we have
		$\displaystyle	\sum_{\textbf{k} \in \mathbb{Z}^r} \Psi_{\rho}(\textbf{x} - \textbf{k}) = 1.$
	\end{rem}

	Next, we present the seminal definitions and properties associated with mixed norm Lebesgue spaces.
	
	\subsection{ Mixed norm Lebesgue space}\label{sec 2.2}
	Let $\overrightarrow{\mathscr{P}}=(p_1,...,p_r),$  where $ p_i \in [1,\infty) $ is an index vector, and $ f : \mathcal{J} \rightarrow \RR $ is a measurable function. Then 
	\begin{equation*} 
		\|f\|_{\overrightarrow{\mathscr{P}}} = \left( \int_{-1}^{1}...\left( \int_{-1}^{1}\left( \int_{-1}^{1} |f(x_1,...,x_r)|^{p_1}dx_1\right) ^{\frac{p_2}{p_1}}dx_2\right) ^{\frac{p_3}{p_2}}... 
		\hspace{4pt} dx_r\right) ^{\frac{1}{p_r}}.
	\end{equation*}
	\begin{defin} (Mixed norm Lebesgue space)
		The mixed norm Lebesgue space $L^{\overrightarrow{\mathscr{P}}}( \mathcal{J})$ for $\overrightarrow{\mathscr{P}}=(p_1,...,p_r)$, $1 \leq p_i < \infty$, is defined as
		$$ L^{\overrightarrow{\mathscr{P}}}( \mathcal{J})= \small\{f:   \mathcal{J} \rightarrow \RR: f \hspace{2pt} \text{is measurable and}\hspace{3pt}  \|f \|_{\overrightarrow{\mathscr{P}}}< \infty \small\}.$$
	\end{defin}
	The \textit{H\"{o}lder's inequality} in the context of mixed norm Lebesgue spaces states that (see \cite{dct})
	\begin{equation*}\label{holder}
		\|f g\|_1\leq \|f\|_{\overrightarrow{\mathscr{P}}}\|g\|_{\overrightarrow{\mathscr{Q}}},
	\end{equation*}
	for $f\in L^{\overrightarrow{\mathscr{P}}}( \mathcal{J}) $ and $ g\in L^{\overrightarrow {\mathscr{Q}}}( \mathcal{J})$. Here $ {\overrightarrow{\mathscr{P}}= (p_1,...,p_r)\in [1,\infty]^{r}},$ ${\overrightarrow {\mathscr{Q}}= (q_1,...,q_r)\in [1,\infty]^{r}}$
	with each pair $(p_k,q_k),$ satisfying $\frac{1}{p_k} +\frac{1}{q_k}=1$, for  $k=1,...,r$. In what follows, we only focus on the tuples $\overrightarrow{\mathscr{P}}=(p_1,p_2,\dots,p_n)$ with non-decreasing components, i.e., $p_i\leq p_{i+1}$ for $1\leq i\leq r-1$.

	
	We now record fundamental definitions and result related to Orlicz spaces, which will be required for the subsequent analysis. 
	
	\subsection{Rudiments of Orlicz space}
	
	\begin{defin} \label{xx}(Convex function)
	A function \( f : I \rightarrow \mathbb{R} \) is said to be \emph{convex} function if
	\[
	f(s x + (1 - s)y) \leq s f(x) + (1 - s)f(y)
	\]
	for every \( x, y \in I \) and \( s \in [0,1] \).
	
	Assume that \( f : [0, \infty) \rightarrow [0, \infty) \) is convex and \( f(0) = 0 \). Choosing \( y = 0 \) in the definition and \( s = \lambda \) or \( s = \frac{1}{\lambda} \), we obtain
	\[
	\begin{aligned}
		f(\lambda x) &\leq \lambda f(x) \quad &&\text{for } \lambda \in [0,1], \\
		f(\lambda x) &\geq \lambda f(x) \quad &&\text{for } \lambda \geq 1.
	\end{aligned}
	\]
	\end{defin}
	\begin{defin}  (Orlicz function)
		An Orlicz function is a convex function  $\phi:[0, \infty)\to [0, \infty]$ such that
		\begin{enumerate}\label{2.3}
			\item [$(\text{1})$] $\phi$ is non-decreasing; 
			
			\item [$(\text{2})$]  $\phi$ is continuous from the left;
			
			\item [$(\text{3})$] $\phi(0)=0$ and $\lim\limits_{u\to \infty} \phi(u)=\infty .$
			
		\end{enumerate}
	\end{defin}
	The notion of modular convergence is essential to study the convergence in the framework of Orlicz space. The functional $I^{\phi}$  associated with an Orlicz function $\phi$ on $M([-1,1])$ is defined as
	\begin{equation*} 
		I^{\phi}[g]:=\int_{-1}^{1} \phi(|g(x)|) dx .
	\end{equation*}
	This functional $I^{\phi}$ extends the classical Lebesgue integral 
	 $\int |g|^p\,dx$  and serves as the foundation for defining Orlicz spaces. It is referred to as a modular functional and satisfies the following properties for  every $g,\,h \in  M([-1,1]):$
	\begin{enumerate}[label=(\roman*)]
		\item $I^{\phi}(g)=0$ if and only if $g=0,$
		\item $I^{\phi}(-g)=I^{\phi}(g),$
		\item $I^{\phi}(\alpha g +\beta h) \leq I^{\phi}(g)+I^{\phi}(h) , \hspace{3pt}for\hspace{3pt} \alpha,\beta \geq 0, \hspace{3pt} \alpha +\beta =1.$
	\end{enumerate}
	Moreover, for any $g \in  M([-1,1])$, the map $\alpha\mapsto I^{\phi}(\alpha g)$ is non-decreasing for $\alpha >0$.
	\begin{defin}(Orlicz space)
		The Orlicz space associated with Orlicz function $\phi$ is defined as
		\begin{equation*}\label{L}
			L^{\phi}([-1,1]):=\left\lbrace f \right. \in M([-1,1]) : I^{\phi}[\lambda f]< \infty \hspace{3pt} \text{for some}\hspace{2pt} \lambda >0 \left.\right\rbrace. 
		\end{equation*}
		It is important to mention that the space $L^{\phi}([-1,1])$ forms a normed linear space with the \emph{Luxemburg norm} denoted by $\|\cdot\|_{\phi}$, defined as 
		$$\|f\|_{\phi} := \inf \{ \lambda > 0 : I^{\phi}[ f/ \lambda] \leq 1 \}.$$
	\end{defin}
	Numerous significant examples of Orlicz functions and the associated Orlicz spaces are described below.
	\begin{enumerate} [label=(\Roman*)., left=0pt]
		\item For \( p \in (0, \infty) \), define $\phi(u)=\displaystyle \frac{u^p}{p}$ and, for $p=\infty,$  set
		\[
		\phi(u) = 
		\begin{cases}
			0, & \text{for }  u \leq 1, \\
			\infty, & \text{for } u > 1.
		\end{cases}
		\]
		This construction leads to the classical Lebesgue spaces $L^p(\RR)$, $1\leq p \leq \infty.$
		
			\item $\phi(u)=(u-1)\Chi_{[1,\infty)}(u)$, this leads to the space $L^1(\RR)+L^{\infty}(\RR)$.
		\item $\phi(u)=u^{\alpha}\log ^{\beta}(u+e)$, for $\alpha \geq 1, \beta >0$ corresponding to the \emph{Logarithmic space} $L^{\alpha}\log ^{\beta}L(\RR)$.

	\end{enumerate}
	
	A family of functions \( (f_k)_{k > 0} \subset L^{\phi}([-1,1]) \) is said to be modularly convergent to \( f \in L^{\phi}([-1,1]) \) if there exists \( \lambda >0 \) such that
	\[
	\lim_{k \to \infty} I^{\phi}[\lambda(f_k - f)] = 0.
	\]
	Also, a family of functions \( (f_k)_{k > 0} \subset L^{\phi}([-1,1]) \) is said to converge in norm to \( f \in L^{\phi}([-1,1]) \) if 
	\[
	\lim_{k \to \infty} I^{\phi}[\lambda(f_k - f)] = 0
	\]
	holds for all \( \lambda > 0 \).
	
	It is important to note that modular and  norm   convergence are coincide if and only if $\phi$-function fulfills the $\Delta_{2}$-condition given by
	$$\phi(2u) \leq N\phi(u), \quad  u \in [0,\infty),$$
	for some $N>0$.

	\begin{defin} (Orlicz conjugate)
		The function $\psi:[0,\infty) \to [0,\infty]$ is an Orlicz conjugate of an Orlicz function $\phi$, is given by
		$$
		\psi(u) \equiv
		\begin{cases}
			\displaystyle\sup_{v \geq 0} \hspace{1pt}(uv -\phi(v) ), & \text{when } u \in [0, \infty), \\
			\infty, & \text{when } u = \infty.
		\end{cases}
		$$
	\end{defin}	
	Let $\psi$ be the conjugate function of the Orlicz function $\phi$. Then, for $f\in L^\phi(\RR^n)$ and $g\in L^{\psi}(\RR^n)$, the \textit{H\"{o}lder's inequality} in Orlicz space is defined as 
	$$\|fg\|_{1} \leq 2 \|f\|_{\phi} \|g\|_{\psi}.$$
	
	\begin{defin}(Jensen's integral inequality) Given a measurable space $(\Omega, \mu)$, 
		the Jensen's integral inequality  is satisfied for every function \( g \in L^{\phi}(\Omega) \) provided the following inequality is holds 
		\[ \phi\left( \frac{1}{\mu(\Omega)}\int_\Omega g(x)\,d\mu\right) \leq \frac{1}{\mu(\Omega)}\displaystyle \int_\Omega \phi(g(x))\,d\mu.\]
	\end{defin}
	To explore more results and details about Orlicz spaces, reader can refer \cite{ k2007, nn,  Raobook, orliczapp, 13, 2025}.
	We are now in a position to extend the framework of Orlicz space to the context of mixed norm spaces.
		\subsection{Mixed norm Orlicz space}
		Let \(\overrightarrow{\Phi} = (\phi_1,..., \phi_r)\), where \(\phi_i\) is an Orlicz function for $ 1 \leq i \leq r$. To establish the modular convergence in the framework of mixed norm Orlicz spaces, the modular functional 
		$I^{\overrightarrow{\Phi}}: M(\mathcal{J}) \rightarrow [0,\infty]$ is defined by 
		\begin{equation*} 
			I^{\overrightarrow{\Phi}}[f] :=\int_{-1}^{1} \phi_{r}\left( \int_{-1}^{1}  \phi_{r-1}\cdots\left(\int_{-1}^{1} \phi_1\left(|f(x_1,...,x_r)|\right)dx_1\right)\cdots dx_{r-1}\right)dx_r\hspace{1pt} .
		\end{equation*}
	\begin{defin}(Mixed norm Orlicz space)
		The Mixed norm Orlicz space generated by $\overrightarrow{\Phi}=(\phi_1,...,\phi_r)$ is defined as
		$$L^{\overrightarrow{\Phi}}(\mathcal{J})=\left\lbrace f \right. \in M(\mathcal{J}) : I^{\overrightarrow{\Phi}}[\lambda f] < \infty \hspace{3pt} \text{for some}\hspace{2pt} \lambda >0 \left.\right\rbrace. $$
	\end{defin}
	Moreover, we say that a family of functions $(f_k)_{k>0} \subset L^{\overrightarrow{\Phi}}(\mathcal{J})$ is  said to be modularly converges to $f \in L^{\overrightarrow{\Phi}}(\mathcal{J})$ if
	$$\lim_{k\to\infty}I^{\overrightarrow{\Phi}}[\lambda(f_k-f)]=0$$
	for some $\lambda >0.$
	In particular, a family of functions $(f_k)_{k>0} \subset L^{\overrightarrow{\Phi}}(\mathcal{J})$ is said to converge in norm to  
	$f \in L^{\overrightarrow{\Phi}}(\mathcal{J}) $ if
	$$\lim_{k\to\infty}I^{\overrightarrow{\Phi}}[\lambda(f_k-f)]=0$$
	for every $\lambda >0.$ 
	We say that $\overrightarrow{\Phi}=(\phi_1,...,\phi_r)$ satisfies $\Delta_{2}$-condition provided that each component function $\phi_i$ satisfies $\Delta_{2}$-condition.
	
	The space $L^{\overrightarrow{\Phi}}(\Omega)$ generalizes the concept of mixed norm Lebesgue space $L^{\overrightarrow{P}}(\Omega)$, as introduced in Subsection \ref{sec 2.2}, paritcularly under the assumption that the components of $\overrightarrow{P}=(p_1,p_2,\dots, p_n)$ are non-decreasing. Specifically, the Orlicz functions are chosen as $\phi_1(x)=x^{p_1}$ and $\phi_i(x)=x^\frac{p_i}{p_{i-1}}$, for $2 \leq i \leq n$, the corresponding mixed norm Orlciz space $L^{\overrightarrow{P}}(\Omega)$ coincides with mixed norm Lebesgue space $L^{\overrightarrow{P}}(\Omega)$.

	\section{Multivariate Kantorovich-type NN Operators in mixed norm Lebesgue Spaces}\label{3}
	For a locally integrable function $f :  \mathcal{J} \rightarrow \RR$, the family of multivariate Kantorovich-type NN operators based on feedforward neural network for $n \in \NN$ is defined as (see \cite{spigler})
	\begin{equation}\label{multi}
		K_n(f,\textbf{x}):= \frac{\displaystyle\sum_{\textbf{k}=-n}^{n-1}  \left(n^r\int_{I_{\textbf{k},n}}f(\textbf{t})d\textbf{t}\right)\Psi_{\rho}(n\textbf{x}-\textbf{k})}{\displaystyle\sum_{\textbf{k}=-n}^{n-1} \Psi_{\rho}(n\textbf{x}-\textbf{k})}
	\end{equation}
	where $ \hspace{5pt} \displaystyle I_{\textbf{k},n}=\prod_{j=1}^{n}\left[\frac{k_j}{n},\frac{k_j+1}{n}\right]$.
	In order to prove the convergence of above family (\ref{multi}) in the framework of mixed norm Lebesgue space $ L^{\overrightarrow{\mathscr{P}}}(\mathcal{J})$, we first establish the following result.
	\begin{thm}\label{thm 3.1}
		Let $f \in L^{\overrightarrow{\mathscr{P}}}(\mathcal{J})$ and $\overrightarrow{\mathscr{P}}=(p_1,...,p_r)$. Then, we have
		$$\|K_nf\|_{\overrightarrow{\mathscr{P}}}\leq  \left(\frac{\|\Psi_{\rho}\|_{1}}{\left(\Psi_{\rho}(2)\right)} \right)^{\sum_{i=1}^{r}\frac{1}{p_i}}  \hspace{3pt}\|f\|_{\overrightarrow{\mathscr{P}}}.$$
	\end{thm}
	\begin{proof}
		For better mathematical visualization, we begin by demonstrating this result for bivariate case, i.e., $r=2$,  $\overrightarrow{\mathscr{P}}=(p_1,p_2)$, and further we will extend it for any $r \geq 3$, $r \in \NN.$
		Since $|\cdot|^{p_1}$ convex, we apply Jensen's inequality. Thus, we have
	\begin{flalign*}
		&\|K_nf\|_{(p_1,p_2)}^{p_2}\nonumber \\
		&= \int_{-1}^{1}\left(\int_{-1}^{1}\left(\frac {\displaystyle\sum_{k_2=-n}^{n-1}   \displaystyle\sum_{k_1=-n}^{n-1} \left|n^2\int_{\frac{k_2}{n}}^{\frac{k_2+1}{n}}\int_{\frac{k_1}{n}}^{\frac{k_1+1}{n}}f(t_1,t_2)dt_1dt_2\right| \Psi_{\rho}(nx_1-k_1,nx_2-k_2)}
		{\displaystyle\sum_{k_2=-n}^{n-1}  \displaystyle\sum_{k_1=-n}^{n-1} \Psi_{\rho}(nx_1-k_1,nx_2-k_2)} \right)^{p_1}dx_1\right)^{\frac{p_2}{p_1}}dx_2\\
		&= \int_{-1}^{1}\left(\int_{-1}^{1}\left(\frac {\displaystyle\sum_{k_2=-n}^{n-1}   \displaystyle\sum_{k_1=-n}^{n-1} \left|n^2\int_{\frac{k_2}{n}}^{\frac{k_2+1}{n}}\int_{\frac{k_1}{n}}^{\frac{k_1+1}{n}}f(t_1,t_2)dt_1dt_2\right|^{p_1} \Psi_{\rho}(nx_1-k_1,nx_2-k_2)}
		{\displaystyle\sum_{k_2=-n}^{n-1}  \displaystyle\sum_{k_1=-n}^{n-1} \Psi_{\rho}(nx_1-k_1,nx_2-k_2)} \right)dx_1\right)^{\frac{p_2}{p_1}}dx_2.
	\end{flalign*}
	Using	Jensen's inequality on integration, Lemma \ref{lem 2.3} and (\ref{we}), we get 
	\begin{flalign*}
		&\|K_nf\|_{(p_1,p_2)}^{p_2}\\
		&\leq \int_{-1}^{1}\left(\int_{-1}^{1}\frac{ \displaystyle\sum_{k_2=-n}^{n-1} \displaystyle\sum_{k_1=-n}^{n-1} \Psi_{\rho}(nx_{1}-k_1,nx_{2}-k_2)\hspace{3pt} n^2 \int_{\frac{k_2}{n}}^{\frac{k_2+1}{n}} \int_{\frac{k_1}{n}}^{\frac{k_1+1}{n}} |f(t_1,t_2)|^{p_1}dt_1dt_2}{\displaystyle \sum_{k_2=-n}^{n-1}\sum_{k_1=-n}^{n-1} \Psi_{\rho}(nx_{1}-k_1,nx_{2}-k_2)}	dx_1 \right)^{\frac{p_2}{p_1}}dx_2 \\
		&\leq \int_{-1}^{1}\left(\int_{-1}^{1}\frac{ \displaystyle\sum_{k_2=-n}^{n-1} \displaystyle\sum_{k_1=-n}^{n-1} \Psi_{\rho}(nx_{1}-k_1)\Psi_\rho(nx_{2}-k_2)\hspace{3pt} n^2 \int_{\frac{k_2}{n}}^{\frac{k_2+1}{n}} \int_{\frac{k_1}{n}}^{\frac{k_1+1}{n}} |f(t_1,t_2)|^{p_1}dt_1dt_2}{\displaystyle \sum_{k_2=-n}^{n-1}\sum_{k_1=-n}^{n-1} \Psi_{\rho}(nx_{1}-k_1)\Psi_\rho(nx_{2}-k_2)}	dx_1 \right)^{\frac{p_2}{p_1}}dx_2 \\
		&\leq \int_{-1}^{1}\left(\int_{-1}^{1}\frac{ \displaystyle\sum_{k_2=-n}^{n-1} \displaystyle\sum_{k_1=-n}^{n-1} \Psi_{\rho}(nx_{1}-k_1)\Psi_\rho(nx_{2}-k_2)\hspace{3pt} n^2 \int_{\frac{k_2}{n}}^{\frac{k_2+1}{n}} \int_{\frac{k_1}{n}}^{\frac{k_1+1}{n}} |f(t_1,t_2)|^{p_1}dt_1dt_2}{\Psi_\rho(2)\displaystyle \sum_{k_2=-n}^{n-1} \Psi_\rho(nx_{2}-k_2)}	dx_1 \right)^{\frac{p_2}{p_1}}dx_2 .
	\end{flalign*}
	Let us assume $nx_1-k_1=u_1$. Then, we have 
	\begin{flalign*}
		\|K_nf\|_{(p_1,p_2)}^{p_2}
		&\leq \int_{-1}^{1}\left(\frac{ \displaystyle\sum_{k_2=-n}^{n-1}  \Psi_\rho(nx_{2}-k_2)\hspace{3pt} n \int_{\frac{k_2}{n}}^{\frac{k_2+1}{n}} \int_{-1}^{1} |f(t_1,t_2)|^{p_1}dt_1dt_2 \int_{\RR}\Psi_{\rho}(u_1)du_1}{\Psi_\rho(2)\displaystyle \sum_{k_2=-n}^{n-1} \Psi_\rho(nx_{2}-k_2)}	 \right)^{\frac{p_2}{p_1}}dx_2 \\
		&\leq  \int_{-1}^{1}\left( \frac{\|\Psi_{\rho}\|_1}{\left(\Psi_{\rho}(2)\right)}\frac{\displaystyle\sum_{k_2=-n}^{n-1}  \Psi_\rho(nx_{2}-k_2)\hspace{3pt} n \int_{\frac{k_2}{n}}^{\frac{k_2+1}{n}} \int_{-1}^{1}|f(t_1,t_2)|^{p_1}dt_1dt_2}{\displaystyle\sum_{k_2=-n}^{n-1}  \Psi_\rho(nx_{2}-k_2)}	\right)^{\frac{p_2}{p_1}}dx_2.
		\end{flalign*}
		Using once again the convexity of the function $|\cdot|^{\frac{p_2}{p_1}}$, it follows from Jensen's inequality and Lemma \ref{lem 2.3} that 
		\begin{flalign*}
		\|K_nf\|_{(p_1,p_2)}^{p_2}	&\leq  \left(\frac{\|\Psi_{\rho}\|_1}{\left(\Psi_{\rho}(2)\right)}\right)^{\frac{p_2}{p_1}}\int_{-1}^{1}\left( \frac{\displaystyle\sum_{k_2=-n}^{n-1}  \Psi_\rho(nx_{2}-k_2)\hspace{3pt} n \int_{\frac{k_2}{n}}^{\frac{k_2+1}{n}} \int_{-1}^{1}|f(t_1,t_2)|^{p_1}dt_1dt_2}{\displaystyle\sum_{k_2=-n}^{n-1}  \Psi_\rho(nx_{2}-k_2)}	\right)^{\frac{p_2}{p_1}}dx_2\\
		&\leq  \frac{\|\Psi_{\rho}\|_1^{\frac{p_2}{p_1}}}{\left(\Psi_{\rho}(2)\right)^{1+\frac{p_2}{p_1}}}  \int_{-1}^{1}\displaystyle\sum_{k_2=-n}^{n-1}\Psi_\rho(nx_{2}-k_2)\hspace{2pt} \left( n \int_{\frac{k_2}{n}}^{\frac{k_2+1}{n}}\int_{-1}^{1}|f(t_1,t_2)|^{p_1}dt_1dt_2\right)^{\frac{p_2}{p_1}}dx_2.
	\end{flalign*}
	Substituting $nx_2-k_2=u_2$ and applying Jensen's inequality on integration, we obtain
	\begin{flalign*}
		\|K_nf\|_{(p_1,p_2)}^{p_2}
		&\leq \frac{\|\Psi_{\rho}\|_1^{\frac{p_2}{p_1}}}{\left(\Psi_{\rho}(2)\right)^{1+\frac{p_2}{p_1}}} \int_{\RR}\displaystyle\sum_{k_2=-n}^{n-1}\Psi_\rho(nx_{2}-k_2)\hspace{2pt}n\int_{\frac{k_2}{n}}^{\frac{k_2+1}{n}}\hspace{2pt} \left(  \int_{-1}^{1}|f(t_1,t_2)|^{p_1}dt_1\right)^{\frac{p_2}{p_1}}dt_2\hspace{2pt}dx_2\\
		&\leq \frac{\|\Psi_{\rho}\|_1^{\frac{p_2}{p_1}}}{\left(\Psi_{\rho}(2)\right)^{1+\frac{p_2}{p_1}}} \displaystyle\sum_{k_2=-n}^{n-1}\hspace{2pt}\int_{\frac{k_2}{n}}^{\frac{k_2+1}{n}}\hspace{2pt} \left( \int_{-1}^{1}|f(t_1,t_2)|^{p_1}dt_1\right)^{\frac{p_2}{p_1}}dt_2 \times  \int_{\RR} \Psi_\rho(u_2)du_2\\
		&\leq \left( \frac{\|\Psi_\rho\|_{1}}{\left(\Psi_{\rho}(2)\right)} \right)^{1+\frac{p_2}{p_1}} \int_{-1}^{1}\hspace{2pt} \left(  \int_{-1}^{1}|f(t_1,t_2)|^{p_1}dt_1\right)^{\frac{p_2}{p_1}}dt_2\\
		&=\left( \frac{\|\Psi_\rho\|_{1}}{\left(\Psi_{\rho}(2)\right)}\right)^{1+\frac{p_2}{p_1}} \hspace{2pt}\|f\|_{(p_1,p_2)}^{p_2}.
	\end{flalign*}
	This gives
	\begin{flalign*}
		\|K_nf\|_{(p_1,p_2)} \leq \left(  \frac{\|\Psi_\rho\|_{1}}{\left(\Psi_{\rho}(2)\right)}\right)^{\frac{1}{p_2}+\frac{1}{p_1}} \hspace{2pt}\|f\|_{(p_1,p_2)}.
	\end{flalign*}
	
	
	Proceeding in a similar manner for higher dimension case, we obtain
	\begin{flalign*}
		\|K_{n}f\|_{\overrightarrow{\mathscr{P}}} 
		\leq \left(\frac{\|\Psi_{\rho}\|_{1}}{\left(\Psi_{\rho}(2)\right)} \right)^{\sum_{i=1}^{r}\frac{1}{p_i}}  \hspace{3pt}\|f\|_{\overrightarrow{\mathscr{P}}}.
	\end{flalign*}

\end{proof}
%
\begin{rem}\label{re1}
	The following condition is equivalent to property (b) of Remark \ref{rem 2.4}
	\begin{equation} \label{eq 3.4}
		\Psi_{\rho}^{\wedge}(k) =
		\begin{cases}
			{1,} &\quad\text{} \ \  {k=0},\\
			{ 0,} &\quad\text{} \ \  {k \neq 0},\\
		\end{cases}
	\end{equation}
	for $k \in \ZZ,$ the Fourier transform of $\Psi_{\rho}$ is given by $\Psi_{\rho}^{\wedge}(v)=\int_{\RR}\Psi_{\rho}(u)e^{-iuv}du$. Consequently, from (\ref{eq 3.4}), it follows that $\int_{\RR}\Psi_{\rho}(u)du=1$. Moreover
	$$\int_{\RR}\cdot \cdot \cdot\int_{\RR}\Psi_{\rho}(\mathbf{u})d\mathbf{u}=\int_{\RR}\cdot \cdot \cdot\int_{\RR}\prod_{i=1}^{r}\Psi_{\rho}(u_i)d\mathbf{u}=\prod_{i=1}^{r}\left(\int_{\RR}\cdot \cdot \cdot\int_{\RR}\Psi_{\rho}(u_i)du_i\right)=1.$$
\end{rem}
Now using Remark \ref{re1}, Theorem \ref{thm 3.1} deduce to 
\begin{flalign*}
	\|K_nf\|_{\overrightarrow{\mathscr{P}}} \leq 	\frac{\|f\|_{\overrightarrow{\mathscr{P}}}}{\left(\Psi_{\rho}(2)\right)^{\sum_{i=1}^{r}\frac{1}{p_i}}}  \hspace{3pt}.
\end{flalign*}

	%
	
	
	The following theorem addresses the convergence of the family \( (K_n f) \) in \( C(\mathcal{J}) \).
	
	\begin{thm}\label{pcgs}
		For $f \in C(\mathcal{J})$ and any $\epsilon>0$, $\exists\hspace{1pt} n\in \NN$ such that
		$$\|K_nf-f\|_{\overrightarrow{\mathscr{P}}} <\epsilon.$$ 
	\end{thm}
	\begin{proof}
		In view of \cite[Theorem 4.1]{spigler}, for $\displaystyle \epsilon_1 :=\frac{\displaystyle \epsilon}{2^{\frac{1}{p_1}+\frac{1}{p_2}}} $, there exists $N_0 \in \NN$ such that
		$	\|K_nf-f\|_{\infty} <\epsilon_1$, \hspace{1pt} for $n\geq N_0.$
		This gives
		\begin{flalign*}
			\|K_nf-f\|_{\overrightarrow{\mathscr{P}}} &=\left(\int_{-1}^{1}\left(\int_{-1}^{1}|(K_nf)(x_1,x_2)-f(x_1,x_2)| dx_1\right)^{\frac{p_2}{p_1}}dx_2\right)^{\frac{1}{p_2}}\\
			&\leq  2^{\left(\frac{1}{p_1}+\frac{1}{p_2}\right)}\|K_nf-f\|_{\infty}\\
			&< \epsilon.
		\end{flalign*}
	\end{proof}
	
	To establish the convergence of the multivariate Kantorovich-type NN operators we utilize the following lemma.
	
	\begin{lemma}\label{vis}
		\cite[Theorem 2.7]{pandey}	The space of continuous functions $C(\mathcal{J})$
		is dense in $L^{\overrightarrow{\mathscr{P}}}(\mathcal{J})$ for $\overrightarrow{\mathscr{P}} \in [1,\infty)^r$.
	\end{lemma}
	
	We now show that $(K_nf)$ converges within mixed norm Lebesgue space $L^{\overrightarrow{\mathscr{P}}}(\mathcal{J})$.
	\begin{thm} \label{thm 3.3}
		Let $f\in L^{\overrightarrow{\mathscr{P}}}(\mathcal{J})$ and for any $\epsilon>0$, $\exists\hspace{1pt} n\in \NN$ such that 
		$$ \|f-K_{n}f\|_{\overrightarrow{\mathscr{P}}}<\epsilon. $$
	\end{thm}
	\begin{proof}
		Let $f\in L^{\overrightarrow{\mathscr{P}}}(\mathcal{J})$ and $\epsilon >0$. In view of Lemma \ref{vis}, there exists $g \in C(\mathcal{J})$ such that 
		$$\|f-g\|_{\overrightarrow{\mathscr{P}}}< \frac{\epsilon}{2\left(\frac{\|f\|_{\overrightarrow{\mathscr{P}}}}{\left(\Psi_{\rho}(2)\right)^{\sum_{i=1}^{r}\frac{1}{p_i}}}  \hspace{3pt}+1\right)}.$$
		Hence using Theorem \ref{thm 3.1}, \ref{pcgs} and the triangle inequality for $\|\cdot\|_{\overrightarrow{\mathscr{P}}}$, we have
		\begin{equation*}
			\begin{split}
				\|f-K_{n}f\|_{\overrightarrow{\mathscr{P}}} &\leq \|f-g\|_{\overrightarrow{\mathscr{P}}}+\|g-K_{n}g\|_{\overrightarrow{\mathscr{P}}}+\|K_{n}g-K_{n}f\|_{\overrightarrow{\mathscr{P}}}\\
				&\leq \left(\frac{\|f\|_{\overrightarrow{\mathscr{P}}}}{\left(\Psi_{\rho}(2)\right)^{\sum_{i=1}^{r}\frac{1}{p_i}}}  \hspace{3pt}+1\right)\|f-g\|_{\overrightarrow{\mathscr{P}}}+\|g-K_{n}g\|_{\overrightarrow{\mathscr{P}}} \\
				& <{\frac{\epsilon}{2}} + \|g-K_{n}g\|_{\overrightarrow{\mathscr{P}}}\\
				& <{\frac{\epsilon}{2}} + {\frac{\epsilon}{2}}=\epsilon,
			\end{split}
		\end{equation*}
		$\forall \hspace{1pt} n \geq N_{0}$, for some $N_{0} \in \NN.$
	\end{proof}

Following this we have studied the convergence of (\ref{multi}) in mixed norm Lebesgue spaces $L^{ \overrightarrow{\mathscr{P}}}(\mathcal{J})$. Now we will generalize this study by extending our results to a more general setting function spaces namely mixed norm Orlicz Space $L^{ \overrightarrow{\Phi}}(\mathcal{J})$ in the following section.
\section{Multivariate Kantorovich-type NN Operators in mixed norm Orlicz Space}\label{4}
This section aims to establish the convergence of  multivariate Kantorovich-type NN operators within the framework of mixed norm Orlicz spaces $L^{ \overrightarrow{\Phi}}(\mathcal{J})$. To this end, we utilize the boundedness of the operator defined in (\ref{multi}) and the denseness of $C(\mathcal{J})$ in   $L^{ \overrightarrow{\Phi}}(\mathcal{J})$.
In the following theorem, we discuss the boundedness of the operator (\ref{multi})
within $L^{ \overrightarrow{\Phi}}(\mathcal{J})$ using modular functional $I^{\overrightarrow{\Phi}}.$
\begin{thm}\label{thm 4.1}
	Let $\overrightarrow{\Phi}=(\phi_1,\phi_2,...,\phi_r)$  and $\lambda >0.$ For $f \in L^{ \overrightarrow{\Phi}}(\mathcal{J})$,  there holds
	\begin{equation}\label{5.1}
		I^{\overrightarrow{\Phi}}(\lambda \hspace{2pt}K_{n}f) \leq  I^{\overrightarrow{\Phi}}(\lambda' f), \quad \text{where } \lambda'=\frac{\lambda}{\left(\Psi_\rho(2)\right)^r}>0.\end{equation}
\end{thm}
\begin{proof}
	We begin by examining  the result for $r=2, i.e., \overrightarrow{\Phi}=(\phi_1,\phi_2)$. Using Jensen's inequality twice, we obtain
	\begin{flalign*}
		&I^{\phi_1,\phi_2}(\lambda K_{n}f)\\
		&= \int_{-1}^{1} \phi_2\left( \int_{-1}^{1} \phi_1\left(\frac{\lambda \displaystyle\sum_{k_2=-n}^{n-1}\displaystyle \sum_{k_1=-n}^{n-1} \Psi_{\rho}(nx_{1}-k_1,nx_{2}-k_2)\hspace{2pt} n^{2} \int_{\frac{k_2}{n}}^{\frac{k_2+1}{n}}\int_{\frac{k_1}{n}}^{\frac{k_1+1}{n}}|f(t_1,t_2)|dt_1dt_2}{\displaystyle \sum_{k_2=-n}^{n-1}\sum_{k_1=-n}^{n-1} \Psi_{\rho}(nx_{1}-k_1,nx_{2}-k_2)} \right) dx_1\right) dx_2\\
		&\leq \int_{-1}^{1}\phi_2\left(\int_{-1}^{1}\frac{ \displaystyle\sum_{k_2=-n}^{n-1} \displaystyle\sum_{k_1=-n}^{n-1} \Psi_{\rho}(nx_{1}-k_1,nx_{2}-k_2)\hspace{2pt} n^2 \int_{\frac{k_2}{n}}^{\frac{k_2+1}{n}} \int_{\frac{k_1}{n}}^{\frac{k_1+1}{n}} \phi_{1}\left(\lambda|f(t_1,t_2)|\right)dt_1dt_2}{\displaystyle \sum_{k_2=-n}^{n-1}\sum_{k_1=-n}^{n-1} \Psi_{\rho}(nx_{1}-k_1,nx_{2}-k_2)}	dx_1 \right)dx_2.
	\end{flalign*} 
	Taking in account  Lemma \ref{lem 2.3} and (\ref{we}), we get
	\begin{flalign*}
			&I^{\phi_1,\phi_2}(\lambda K_{n}f)\\
		&\leq  \int_{-1}^{1}\phi_2\left( \int_{-1}^{1}\frac{  \displaystyle\sum_{k_2=-n}^{n-1}\sum_{k_1=-n}^{n-1} \Psi_{\rho}(nx_1-k_1)\Psi_\rho(nx_{2}-k_2) \hspace{2pt} n^2  \int_{\frac{k_2}{n}}^{\frac{k_2+1}{n}}\int_{\frac{k_1}{n}}^{\frac{k_1+1}{n}}\phi_{1}\left(\lambda|f(t_1,t_2)|\right)dt_1dt_2}{\displaystyle\sum_{k_1=-n}^{n-1} \sum_{k_2=-n}^{n-1} \Psi_{\rho}(nx_{1}-k_1) \Psi_\rho(nx_{2}-k_2)}dx_1	\right)dx_2 \\
		&\leq  \int_{-1}^{1}\phi_2\left( \int_{-1}^{1} \frac{\displaystyle\sum_{k_2=-n}^{n-1} \sum_{k_1=-n}^{n-1} \Psi_{\rho}(nx_1-k_1)\Psi_\rho(nx_{2}-k_2)\hspace{2pt} n^2 \int_{\frac{k_2}{n}}^{\frac{k_2+1}{n}} \int_{\frac{k_1}{n}}^{\frac{k_1+1}{n}}\phi_{1}\left(\lambda|f(t_1,t_2)|\right)dt_1dt_2}{\left(\Psi_{\rho}(2)\right)\displaystyle\sum_{k_2=-n}^{n-1}  \Psi_\rho(nx_{2}-k_2)}dx_1	\right)dx_2.
	\end{flalign*}
	Let us assume $nx_1-k_1=u_1$ along with using Jensen's inequality, (\ref{ce}) and Lemma \ref{lem 2.3}, we obtain
	\begin{flalign*}
		&	I^{\phi_1,\phi_2}(\lambda K_{n}f)\\
		&\leq    \int_{-1}^{1}\phi_2\left( \frac{\displaystyle\sum_{k_2=-n}^{n-1}  \Psi_\rho(nx_{2}-k_2)\hspace{2pt} n \int_{\frac{k_2}{n}}^{\frac{k_2+1}{n}} \int_{-1}^{1}\phi_{1}\left(\lambda|f(t_1,t_2)|\right)dt_1dt_2 \times \int_{\RR} \Psi_{\rho}(u_1) du_1}{\left(\Psi_{\rho}(2)\right)\displaystyle\sum_{k_2=-n}^{n-1}  \Psi_\rho(nx_{2}-k_2)}	\right)dx_2\\
		&\leq  \frac{1}{\left(\Psi_{\rho}(2)\right)}  \int_{-1}^{1}\displaystyle\sum_{k_2=-n}^{n-1}\Psi_\rho(nx_{2}-k_2)\hspace{2pt} \phi_2\left( \frac{n}{\left(\Psi_{\rho}(2)\right)} \int_{\frac{k_2}{n}}^{\frac{k_2+1}{n}}\int_{-1}^{1}\phi_{1}\left(\lambda|f(t_1,t_2)|\right)dt_1dt_2\right)dx_2.
	\end{flalign*}
	Applying the Jensen's inequality on integration, (\ref{ce}) and setting $nx_2-k_2=u_2$, we get
	\begin{flalign*}
		&	I^{\phi_1,\phi_2}(\lambda K_{n}f)\\
		&\leq \frac{1}{\left(\Psi_{\rho}(2)\right)}  \int_{-1}^{1}\displaystyle\sum_{k_2=-n}^{n-1}\Psi_\rho(nx_{2}-k_2)\hspace{2pt}n\int_{\frac{k_2}{n}}^{\frac{k_2+1}{n}}\hspace{2pt} \phi_2\left( \frac{1}{\left(\Psi_{\rho}(2)\right)} \int_{-1}^{1}\phi_{1}\left(\lambda|f(t_1,t_2)|\right)dt_1\right)dt_2\hspace{2pt}dx_2\\
		&\leq \frac{1}{\left(\Psi_{\rho}(2)\right)}  \displaystyle\sum_{k_2=-n}^{n-1}\hspace{2pt}\int_{\frac{k_2}{n}}^{\frac{k_2+1}{n}}\hspace{2pt} \phi_2\left( \frac{1}{\left(\Psi_{\rho}(2)\right)} \int_{-1}^{1}\phi_{1}\left(\lambda|f(t_1,t_2)|\right)dt_1\right)dt_2 \times \int_{\RR}\Psi_\rho(u_2)du_2\\
			&\leq \frac{1}{\left(\Psi_{\rho}(2)\right)}  \hspace{2pt}\int_{-1}^{1}\hspace{2pt} \phi_2\left( \frac{1}{\left(\Psi_{\rho}(2)\right)} \int_{-1}^{1}\phi_{1}\left(\lambda|f(t_1,t_2)|\right)dt_1\right)dt_2.
		\end{flalign*}
		Since $\Psi_\rho(2) \geq 1$, using Definition \ref{xx}, we can write
		\begin{flalign*}
				I^{\phi_1,\phi_2}(\lambda K_{n}f)	&\leq   \int_{-1}^{1}\hspace{2pt} \phi_2\left(  \int_{-1}^{1}\phi_{1}\left(\frac{\lambda|f(t_1,t_2)|}{\left(\Psi_{\rho}(2)\right)^2}\right)dt_1\right)dt_2\\
		&=  I^{\phi_1,\phi_2}(\lambda' f), \quad \text{where } \lambda'=\frac{\lambda}{\left(\Psi_\rho(2)\right)^2}>0.
	\end{flalign*}
	Proceeding in a similar manner for $\overrightarrow{\Phi}=(\phi_1,...,\phi_r)$, we obtain
	\begin{equation*}
		\begin{split}
			I^{\overrightarrow{\Phi}}(\lambda
			K_nf)
			&\leq I^{\overrightarrow{\Phi}}(\lambda' f), \quad \text{where } \lambda'=\frac{\lambda}{\left(\Psi_\rho(2)\right)^r}>0.
		\end{split}
	\end{equation*}
	%
	Accordingly, the desired result is obtained.		
\end{proof}

	
	
	\begin{rem}\label{rem 5.2}
		The modular inequality on multivariate Kantorovich-type NN operator can be written in terms of norm with respect to mixed norm Orlicz space. Consider $f\in L^{\overrightarrow{\Phi}}(\mathcal{J})$ be arbitrary and $ \lambda=\left\|\frac{f}{\left(\Psi_\rho(2)\right)^r} \right\|_{\overrightarrow{\Phi}}$, then by the definition of mixed norm Orlicz space we have 
		$$I^{\overrightarrow{\Phi}}\Big(\frac{ f}{ \lambda \left(\Psi_\rho(2)\right)^r }\Big)\leq 1.$$ 
		The inequality \eqref{5.1} implies
		\begin{align*}
			I^{\overrightarrow{\Phi}}\Big(\frac{K_nf}{\lambda}\Big)\leq I^{\overrightarrow{\Phi}}\Big(\frac{ f}{\lambda \left(\Psi_\rho(2)\right)^r }\Big)\leq 1.
		\end{align*}
		Therefore, by the definition of mixed  norm Orlicz space, we have
		\begin{align*}
			\|K_nf\|_{\overrightarrow{\Phi}}\leq \left\| \frac{f}{\left(\Psi_\rho(2)\right)^r}\right\|_{\overrightarrow{\Phi}}. 
		\end{align*}
		By definition of norm $\|\cdot\|_{\overrightarrow{\Phi}}$, we can write
		$$\|K_nf\|_{\overrightarrow{\Phi}}\leq \frac{\left\| f\right\|_{\overrightarrow{\Phi}}}{\left(\Psi_\rho(2)\right)^r} . $$
	\end{rem}

To establish the convergence in \( L^{\overrightarrow{\Phi}}(\mathcal{J}) \) for multivariate Kantorovich-type  NN operators, we  prove the following result.
\begin{thm}\label{thm 4.2}
	Let $f \in C(\mathcal{J})$ and $ \lambda > 0 $. Then the following holds
	$$\lim_{n \to +\infty} I^{\overrightarrow{\Phi}}[\lambda(K_nf - f)] = 0.
	$$	
\end{thm}
\begin{proof}
	Following a similar approach as in \cite[Theorem 4.1]{spigler}, for any fixed $\epsilon>0$, we have
	\begin{flalign*}
		I^{\overrightarrow{\Phi}}\left[\lambda \left(K_nf - f\right)\right] &=\int_{-1}^{1}\phi_2\left(\int_{-1}^{1}\phi_1 \left(\lambda\hspace{2pt} |(K_nf)(x_1,x_2)-f(x_1,x_2)|\right)dx_1\right)dx_2	\\
		&\leq 4\hspace{2pt}\phi_2\left(\phi_1 \left(\lambda \|K_nf-f\|_{\infty}\right)\right)\\
		& \leq 4\hspace{2pt}\phi_2\left(\phi_1\left(\lambda \epsilon\right)\right),
	\end{flalign*}
	for sufficiently large $n \in \NN$. In view of definition of Orlicz function $\phi$ and arbitrariness of $\epsilon$, we get the desired result.
\end{proof}

It is well known that within the structure of mixed norm, the space of continuous function \( C(\mathcal{J})\) is dense in \( L^{\overrightarrow{\mathscr{P}}}(\mathcal{J}) \), for \( \overrightarrow{\mathscr{P}} \in [1, \infty)^r \).
According to the existing literature, this result is not explicitly studied in the framework of mixed norm Orlicz spaces. To provide a thorough exposition we address the density result for \( L^{\overrightarrow{\Phi}}(\mathcal{J}) \), in the following lemma. 

\begin{lemma}\label{lem 4.2}
	The space $C(\mathcal{J})$
	is  dense in $L^{\overrightarrow{\Phi}}(\mathcal{J})$.
\end{lemma}
	

\begin{proof}
	Here we prove the result for $r=2.$ The result is demonstrated by proceeding through the following steps:
	
\textbf{Step 1.}	We establish that there exists a sequence of continuous functions that converges to any simple function.
For a measurable set $M \subset \mathcal{J}$ and any $m \in \NN,$ one can  construct a sequence of open sets $(D_m)$ and compact sets $(C_m)$ respectively such that $C_m \subset M \subset D_m$ and $\mu(D_m \setminus C_m)<1/m.$ By applying Urysohn's lemma, we construct a sequence of continuous function $f_m: [-1,1]^m \rightarrow [-1,1]\subset \RR$ such that
	$$f_m(\textbf{x}) = 
	\begin{cases} 
		1, & \text{when } \textbf{x} \in C_m, \\
		0, & \text{when } \textbf{x} \in [-1, 1]^r \setminus D_m.
	\end{cases}$$
	One can directly confirm that $\chi_{C_m}(\textbf{x})\leq  f_m(\textbf{x}) \leq  \chi_{D_m}(\textbf{x})$ for all $\textbf{x} \in \mathcal{J}.$ This shows $0 \leq |f_m(\textbf{x})-\chi_{C_m}(\textbf{x})|\leq \chi_{D_m \setminus C_m}(\textbf{x})$ for all $\textbf{x} \in \mathcal{J}.$ For every $\bar \lambda>0$, one can write
	$$I^{\overrightarrow{\Phi}} (\bar\lambda(\chi_M - f_m ))  
	\leq I^{\overrightarrow{\Phi}}(\bar\lambda(\chi_{D_m} - \chi_{C_m} ))  
	\leq I^{\overrightarrow{\Phi}}(\bar\lambda( \chi_{D_m \setminus C_m} )).$$
	Since $|\chi_{D_m\setminus C_m}(\textbf{x})| \leq 1$ for every $\textbf{x} \in \mathcal{J}$ and it approaches to $zero$ almost everywhere, by using the dominated convergence theorem for mixed norm (\cite[Section 2]{dct}), we conclude that $ I^{\overrightarrow{\Phi}}(\bar\lambda( \chi_{D_m \setminus C_m} )) $  approaches to $zero$   as $m \rightarrow \infty. $ This shows that $I^{\overrightarrow{\Phi}} (\bar\lambda(\chi_M - f_m )) $  approaches to $zero$  as $m \to \infty.$ 
	
	Consequently, given $\epsilon>0$ and a simple function $f : \mathcal{J} \rightarrow \RR$, one can find a continuous function 
 $h: \mathcal{J} \rightarrow\RR$  such that
	\begin{equation}\label{p1}
		\|f-h\|_{\overrightarrow{\Phi}}<\frac{\epsilon}{2}.
	\end{equation}

\textbf{Step 2.}	We  demonstrate that every function in $L^{\overrightarrow{\Phi}}(\mathcal{J})$ can be approximated arbitrarily well by simple functions.
	For $g\in L^{\overrightarrow{\Phi}}(\mathcal{J}),$ there exists an increasing sequence of simple functions $(g_n)_{n \in \NN}$   such that $0\leq g_n \uparrow |g|$ pointwise  as $n \rightarrow \infty$.  Furthermore, for any $\eta>0$, it holds that $\phi_1(\eta(|g|-g_n))\leq \phi_1(\eta g)$ and $\phi_1(\eta(|g|-g_n))$ approaches to zero as $n \rightarrow \infty.$ Hence, by the \emph{dominated convergence theorem} for mixed norm, we have
	$$ \lim_{n\to \infty} \int_{-1}^{1}\phi_2\left(\int_{-1}^{1}\phi_1\left( \eta|(|g|-g_n)(x_1,x_2)| \right)dx_1\right)dx_2 = 0.$$ Therefore, for any positive function $|g|\in L^{\overrightarrow{\Phi}}(\mathcal{J})$ there exists a non-negative simple function arbitrary close to $|g|$.
	As $g\in L^{\overrightarrow{\Phi}}(\mathcal{J})$, then we can write \( g = g^+ - g^- \), where \( g^+, g^- \in L^{\overrightarrow{\Phi}}(\mathcal{J}) \) are non-negative functions defined as follows:
	\begin{equation*}
		g^+(x)=\max\{g(x), 0\} \qquad \text{and} \qquad g^-(x)=-\min\{g(x), 0\}.
	\end{equation*}
	Hence, for any \( \epsilon > 0 \) there exist non-negative simple functions $s^+$ and $s^-$ such that
	\[
	\|g^+ - s^+\|_{\overrightarrow{\Phi}} < \frac{\epsilon}{4}, \quad \|g^- - s^-\|_{\overrightarrow{\Phi}} < \frac{\epsilon}{4}.
	\]
Therefore, for any  $\epsilon>0$ and $g\in L^{\overrightarrow{\Phi}}(\mathcal{J}) $, one can find a simple function \( s := s^+ - s^- \) such that
	\begin{equation}\label{p2}
		\|g - s\|_{\overrightarrow{\Phi}} \leq \|g^+ - s^+\|_{\overrightarrow{\Phi}} + \|g^- - s^-\|_{\overrightarrow{\Phi}} < \frac{\epsilon}{2}.
	\end{equation}
	In view of the triangle inequality of mixed norm Orlicz space, and the equations \eqref{p1} and \eqref{p2}, we get $\|f-g\|_{\overrightarrow{\Phi}}  <\epsilon.$
This completes the proof. Following a similar approach one can extend this for any $r \in \NN.$
\end{proof}

We now proceed to prove the convergence of (\ref{multi}) in  $L^{\overrightarrow{\Phi}}(\mathcal{J})$.
\begin{thm} \label{interpol}
	Let $f\in L^{\overrightarrow{\Phi}}(\mathcal{J})$ be fixed. Then we have
	$$ \lim_{n\to\infty}	\| K_nf-f\|_{\overrightarrow{\Phi}}=0. $$
\end{thm}
\begin{proof}
	For $f\in L^{\overrightarrow{\Phi}}(\mathcal{J})$ and $\epsilon >0$. Then by Lemma \ref{lem 4.2}, there exists $g \in C(\mathcal{J})$  such that
	\begin{equation}\label{Ae}
		\|f-g\|_{\overrightarrow{\Phi}} <\frac{\epsilon}{2 \left( 1+\frac{1}{\left(\Psi_\rho(2)\right)^r}\right)}.
	\end{equation}
	Theorem \ref{thm 4.2} implies there exists $N_0\in \NN$ such that for each $n\geq N_0$, we have $$\|K_ng-g\|_{\overrightarrow{\Phi}}<\frac{\epsilon}{2}.$$
	Now for $n\geq N_0$, the triangle inequality of $ \|\cdot\|_{\overrightarrow{\Phi}}$ along with \eqref{Ae} and Remark \ref{rem 5.2}, we obtain
	\begin{align*}
		\| K_nf-f\|_{\overrightarrow{\Phi}} &\leq \|g-f\|_{  \overrightarrow{\Phi}}+ \|K_nf-K_ng\|_{\overrightarrow{\Phi}}+ \|K_ng-g\|_{\overrightarrow{\Phi}}\\
		&\leq \|g-f\|_{  \overrightarrow{\Phi}}+ \|K_nf-K_ng\|_{\overrightarrow{\Phi}}+ \frac{\epsilon}{2}\\
			&\leq \left( 1+\frac{1}{\left(\Psi_\rho(2)\right)^r}\right) \|g-f\|_{  \overrightarrow{\Phi}}+  \frac{\epsilon}{2}\\
		&<\epsilon.
	\end{align*}
	This completes the proof.
\end{proof}

\begin{rem}
	As discussed in Section \ref{2}, it is clear that $L^{\phi}$ 
	generates various function spaces depending on the choice of
	$\phi-$functions. For an example,
	$\phi_{\alpha, \beta}(x) = x^{\alpha} \log^{\beta}(e + x)$ for $\alpha \geq 1$, $\beta > 0$ and $x \geq 0$. 
	The corresponding Orlicz space is an interpolation space.
	It is denoted by \( L^{\alpha} \log^{\beta} L\).
	The Logarithmic spaces \( L^{\alpha} \log^{\beta} L\) are associated with the Hardy-Littlewood maximal functions (see \cite{log2}). 
	
	Another example is $\phi_{\alpha}(x) = e^{x^{\alpha}} - 1$
	for \( \alpha > 0 \) and \( x \geq 0 \) which generates an Orlicz space known as an exponential space. In Sobolev space embeddings, exponential spaces are of fundamental importance (\cite{exp}).
	
	Here, we consider $(\phi_1,\phi_2)=(\phi_{\alpha, \beta},\phi_{\alpha})$ and define the modular space as the set of functions 
	\( f \in {M}(\mathcal{J}) \) for which
	\[
	I^{(\phi_1,\phi_2)}[\lambda f] =\int_{-1}^{1} \left\lbrace  \exp \left( \int_{-1}^{1} (\lambda |f(\textbf{x})|)^{\alpha} \log^{\beta}(e + \lambda |f(\textbf{x})|) \, dx_1\right)^{\alpha}-1 \right\rbrace  dx_2 < +\infty,
	\]
	for \( \lambda > 0 \). From Theorem \ref{interpol}, we can conclude the following corollaries, particularly $\alpha=\beta=1.$
\end{rem}
\begin{cor}\label{c1}
	Let $f \in L^{(\phi_{\alpha,\beta},{\phi_\alpha})}(\mathcal{J})$ and $\lambda>0$. Then we have 
	\[
	\lim_{n \to +\infty}\int_{-1}^{1} \left\lbrace  \exp \left( \int_{-1}^{1} (\lambda |(K_{n}f)(\textbf{x})-f(\textbf{x})|) \log(e + \lambda |(K_{n}f)(\textbf{x})-f(\textbf{x})|) \, dx_1\right)-1 \right\rbrace  dx_2 = 0.
	\]
\end{cor}

Next, we consider $(\phi_1,\phi_2)=(\phi_{\alpha},\phi_{\alpha, \beta})$.  The corresponding modular space includes functions $f \in {M}(\mathcal{J})$ for which
\[
I^{(\phi_1,\phi_2)}[\lambda f] = \int_{-1}^{1}\left\lbrace \left(\int_{-1}^{1} \left( \exp(\lambda |f(\textbf{x})|^{\alpha}) - 1 \right) \, dx_1 \right)^{\alpha}\log^{\beta}\left(e+\int_{-1}^{1} \left( \exp(\lambda |f(\textbf{x})|^{\alpha}) - 1 \right) \, dx_1\right)\right\rbrace dx_2
\]
is finite	for $\lambda > 0$.

\begin{cor} \label{c2}
	For every \(f \in L^{(\phi_{\alpha},{\phi_{\alpha,\beta}})}(\mathcal{J})\), and \(\lambda > 0\) such that
	\begin{flalign*}
		&\lim_{n\to \infty} \int_{-1}^{1}\left\lbrace \left(\int_{-1}^{1} \left( \exp(\lambda |(K_{n}f)(\textbf{x})-f(\textbf{x})|) - 1 \right) \, dx_1 \right)\right.\\
		&\quad\quad\quad\left.\log\left(e+\int_{-1}^{1} \left( \exp(\lambda |(K_{n}f)(\textbf{x})-f(\textbf{x})|) - 1 \right) \, dx_1\right)\right\rbrace dx_2=0.
	\end{flalign*}
\end{cor}
Since \( \phi_{\alpha} \) does not satisfies the \( \Delta_2 \)-condition, it follows that in Corollary \ref{c1} and \ref{c2}, modular and norm convergence are not equivalent.

\section{examples of activation functions and applications}\label{5}

The sigmoidal function plays an essential role as the activation function in neural networks.
In this section, we provide some numerical examples to support the theoretical results. To demonstrate how multivariate Kantorovich-type neural network operators approximates multivariate functions, we select a two-variable function and show the corresponding fitting surface of graphs, and estimates of error in approximation.

First, we discuss a few examples of smooth sigmoidal functions used as activation functions satisfying (S1)–(S2) discussed in Section \ref{2}. 
An important example in this direction is the \emph{logistic function} $\rho_l$ (see \cite{Anas3, NN}) given by
\[
\rho_l(x) =(1 + e^{-x})^{-1}, \quad x \in \mathbb{R}.
\] 
It is to note that the logistic function satisfies conditions (S1) and (S2) and also possesses Lipschitz continuity. 

Another example in this direction is the \emph{hyperbolic tangent function} $\rho_h$ (see \cite{NN, Anas2, Anas3}), defined as 
\[
\rho_h(x) =2e^x (e^x + e^{-x})^{-1}, \quad x \in \mathbb{R}.
\]
The hyperbolic tangent function have been playing a significant role as activation functions. For instance, as shown in  \cite{mishraNN} that feedforward neural network (FNN) using the hyperbolic tangent activation function, even with only two hidden layers, can approximate functions with similar approximation capability as deeper ReLU-based neural networks.



The \emph{ramp function } $\rho_R$ (see \cite{NN, 2}) is a suitable choice for a non-smooth sigmoidal function and is defined as 
\[
\rho_R(x) =
\begin{cases}
	0, & \text{if } x < -\frac{1}{2}, \\
	x + \frac{1}{2}, & \text{if } -\frac{1}{2} \leq x \leq \frac{1}{2}, \\
	1, & \text{if } x > \frac{1}{2}.
\end{cases}
\]
As the density function corresponding to the ramp function has compact support, it follows that the absolute moment of any order is finite.


A constructive example of sigmoidal function associated with the B-spline is defined by
$$ \rho_{B_n}(x)= \int_{-\infty}^{x}B_n(t)dt,$$

where central B-spline of order $n \in \NN$ is defined as (see \cite{k2007})
\[
B_n(x) := \frac{1}{(n - 1)!} \sum_{i=0}^{n} (-1)^{i} \binom{n}{i} \left(\frac{n}{2} + x - i\right)^{r-1}_+, \quad x \in \RR,
\]
where \((x)_{+} := \max\{x, 0\}\) represents the positive part of \(x \in \mathbb{R}\).
Also, the Fourier transform of \(B_n\) is given by
\[
\hat{B}_n(v) = \text{sinc}^n\left( \frac{v}{2\pi} \right), \quad v \in \mathbb{R}.
\]
 The functions \(B_n\) defined for all \(n \in \mathbb{N}\) are bounded and have compact support within the interval  \([-n/2, n/2]\) on \(\mathbb{R}\). As a result, we have \(B_n \in L^1(\mathbb{R})\), and \(A_{\nu}(\Psi_\rho) < +\infty\) holds.

To demonstrate the practical performance of the proposed neural network operators given in~\eqref{multi}, we present graphical illustrations and error estimations. Specifically, we exhibit their effectiveness in reconstructing bivariate functions using different types of sigmoidal activation functions. In addition, we explore their applications in image processing tasks, including image reconstruction, denoising, scaling, as well as image inpainting. Furthermore, we also discuss the application of mixed norm structures in signal and image processing.

	\subsection{ Graphical  analysis}
	Here, we present some numerical examples to  demonstrate the approximation capabilities of the operator (\ref{multi}) for $r=2$. 
	
	\textbf{Example 1.} Consider the function $f : [0,1] \times [0,1] \rightarrow  \RR$ defined as $
	f(u_1, u_2) = \exp\left(-70\left[(u_1 - 0.3)^2 + (u_2 - 0.5)^2\right]\right) + \exp\left(-70\left[(u_1 - 0.7)^2 + (u_2 - 0.5)^2\right]\right).$
The convergence behavior and approximation accuracy of the operator (\ref{multi}) are illustrated in Figures~\ref{g1} and~\ref{g2}. Furthermore, Table~\ref{t1} and  \ref{or1} shows the numerical approximation errors of the function \( f \) by \( (K_nf) \) in the setting of mixed norm Lebesgue space and mixed norm Orlicz space.  These results clearly demonstrate that the operator  (\ref{multi}) provides an effective approximation to \( f \), with improved accuracy as the parameter \( n \) increases.

	\begin{figure}[htbp]
		\centering
		
		\begin{subfigure}{0.3\textwidth}
			\centering
			\includegraphics[width=\linewidth]{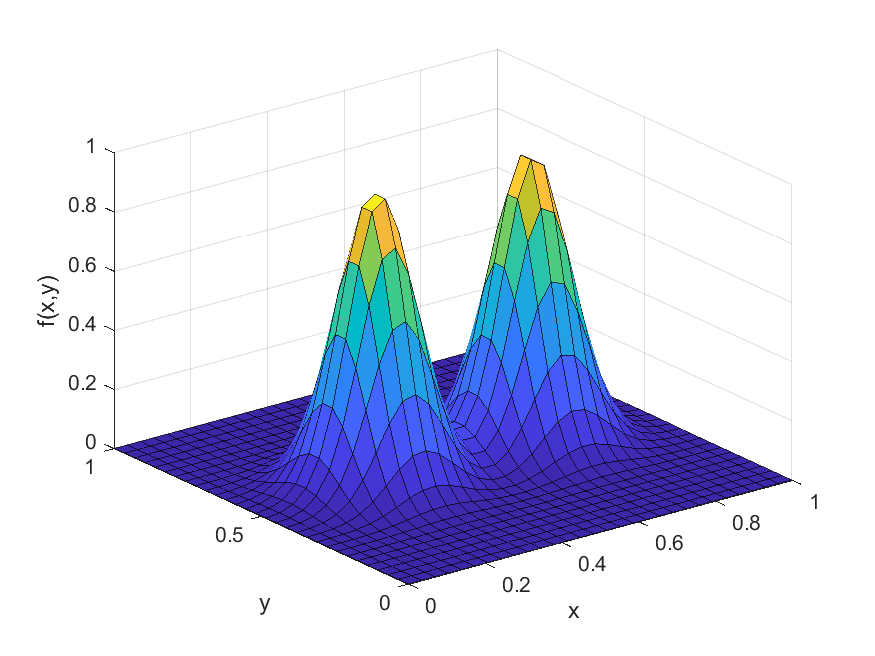}
			\caption{Original function}
		\end{subfigure}
		\hfill
		\begin{subfigure}{0.3\textwidth}
			\centering
			\includegraphics[width=\linewidth]{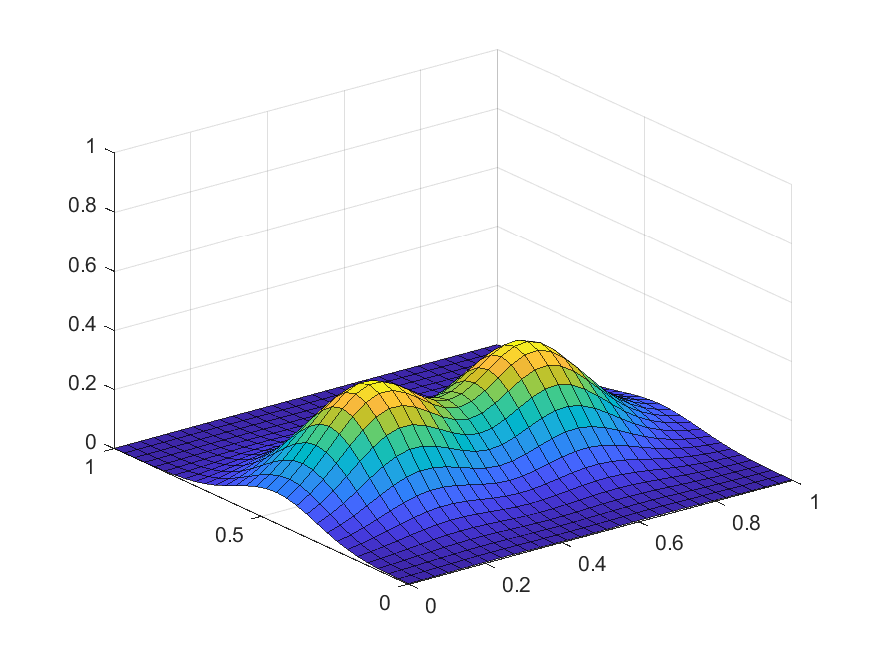}
			\caption{$K_{10} f$}
		\end{subfigure}
		\hfill
		\begin{subfigure}{0.3\textwidth}
			\centering
			\includegraphics[width=\linewidth]{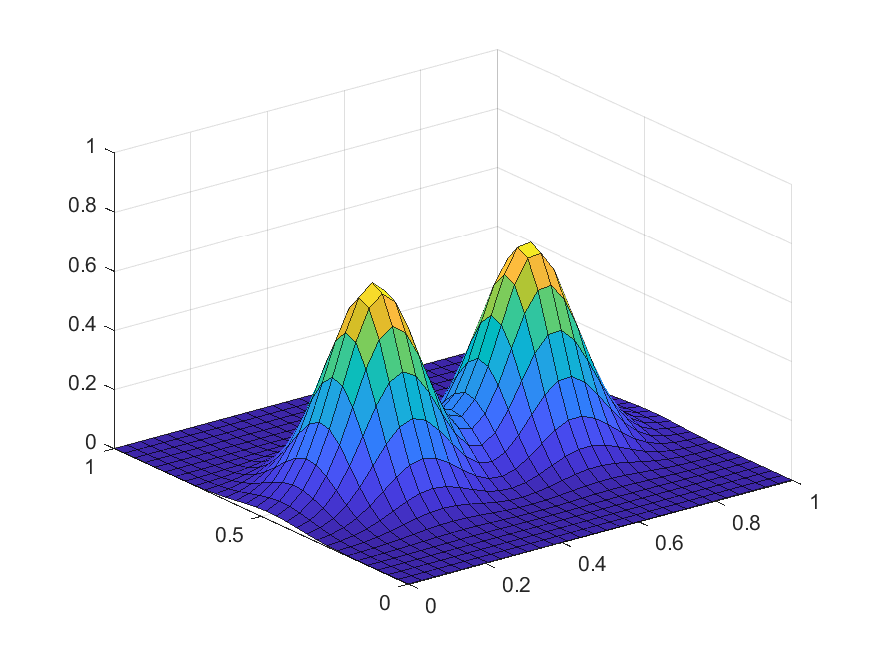}
			\caption{$K_{20} f$}
		\end{subfigure}
		
		\vspace{1em}
		
		\begin{subfigure}{0.3\textwidth}
			\centering
			\includegraphics[width=\linewidth]{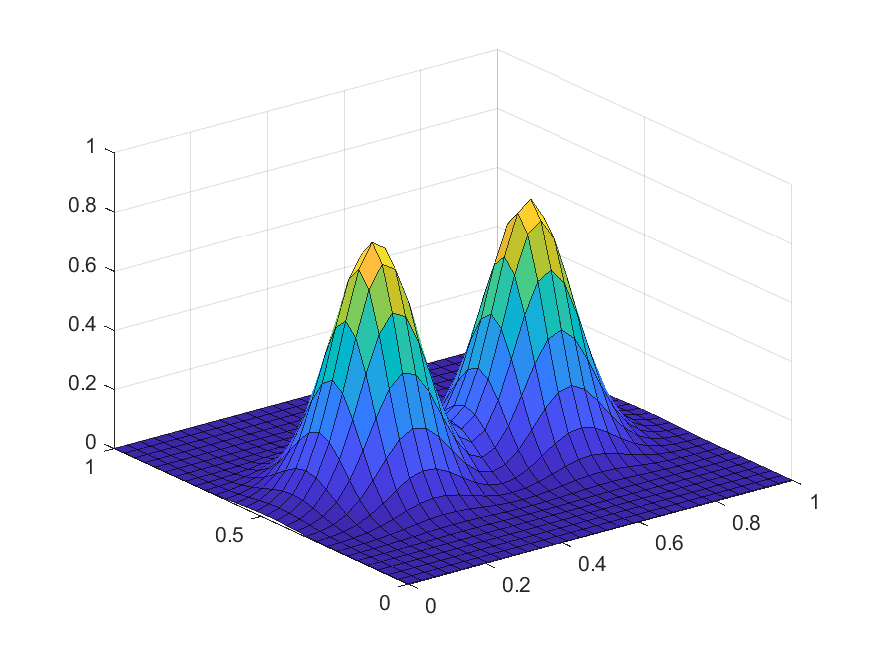}
			\caption{$K_{30} f$}
		\end{subfigure}
		\hfill
		\begin{subfigure}{0.3\textwidth}
			\centering
			\includegraphics[width=\linewidth]{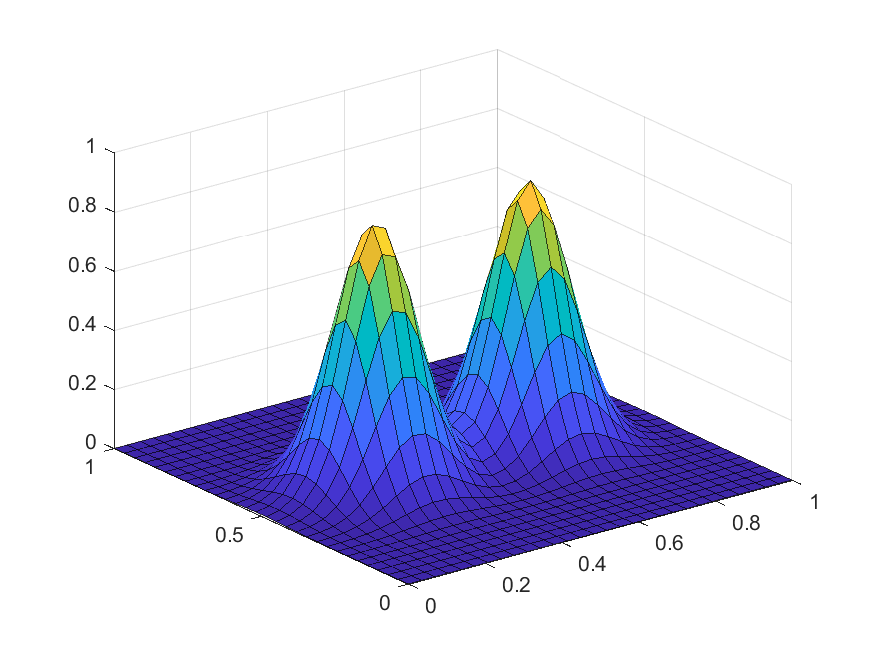}
			\caption{$K_{40} f$}
		\end{subfigure}
		\hfill
		\begin{subfigure}{0.3\textwidth}
			\centering
			\includegraphics[width=\linewidth]{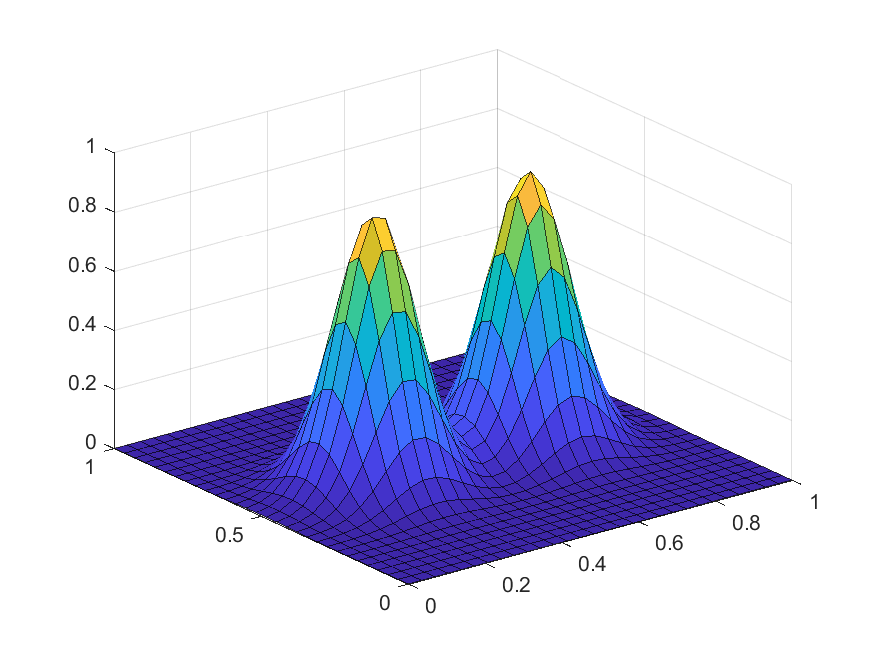}
			\caption{$K_{50} f$}
		\end{subfigure}

		\caption{Approximation of $f$ by $(K_{n} f)$ based on $\rho_h$ for different values of $n$}
		\label{g1}
	\end{figure}

	\begin{table}[h!]
		\centering
		\renewcommand{\arraystretch}{1.3}
		\setlength{\tabcolsep}{12pt}
		\caption{Error-estimates for  modular convergence of $f$ by $(K_n f)$.}
		\vspace{0.5em}
		\label{or1}
		\rowcolors{2}{gray!10}{white}
		\resizebox{0.75\textwidth}{!}{ 
			\begin{tabular}{c c c c c}
				\toprule
				\rowcolor{gray!25}
				$n$ 
				& \multicolumn{2}{c}{$L^{\left(\phi_\alpha,\phi_{\alpha,\beta}\right)}$} 
				& \multicolumn{2}{c}{$L^{\left(\phi_{\alpha,\beta},\phi_\alpha\right)}$} \\
				\cmidrule(lr){2-3} \cmidrule(lr){4-5}
				\rowcolor{gray!25}
				& \makecell{$\alpha = 1.3,$\\$\beta = 1$} 
				& \makecell{$\alpha = 2,$\\$\beta = 1$}
				& \makecell{$\alpha = 1.8,$\\$\beta = 2.2$} 
				& \makecell{$\alpha = 2,$\\$\beta = 1.7$} \\
				\midrule
				10  & 0.027429 & 0.001996 & 0.005597 & 0.002440 \\
				20 & 0.009345 & 0.000175 &  0.000748 & 0.000218 \\
				30 & 0.004903 & 0.000037 &0.000204 & 0.000044\\
				40 & 0.003263 &0.000014 & 0.000091& 0.000016 \\
				\bottomrule
			\end{tabular}
		}
	\end{table}

		\begin{table}[h!]
		\centering
		\renewcommand{\arraystretch}{1.2}
		\setlength{\tabcolsep}{10pt}
		\caption{Max absolute error and Mixed $L^{(p_1,p_2)}$-error of approximation of   $(K_nf)$ to $f$ based on $\rho_h$}
		\label{t1}
		\vspace{0.5em}
		\rowcolors{2}{gray!10}{white}
		\begin{tabular}{>{\bfseries}c c c c c}
			\toprule
			\rowcolor{gray!25}
			$n$ & Max absolute error & Mixed $L^{2,3}$-error & Mixed $L^{4,6}$-error & Mixed $L^{3,7}$-error \\
			\midrule
			10 &  0.66529 & 0.16642 & 0.29694 & 0.28452 \\
			20 & 0.33918 & 0.097045 & 0.16852 & 0.16169 \\
			30 & 0.18669 & 0.067351 &  0.11239 & 0.10835 \\
			40 & 0.1147 & 0.053271 & 0.087012 & 0.084134 \\
			\bottomrule
		\end{tabular}
	\end{table}

	\begin{figure}[H]
	\centering
	\begin{subfigure}[t]{0.54\textwidth}
		\includegraphics[width=\linewidth]{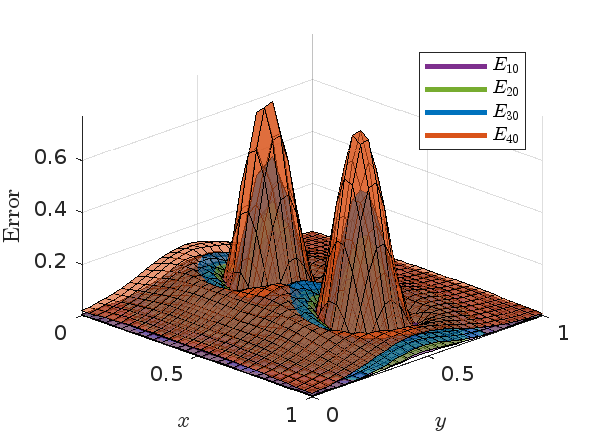}  
		\caption{Error of approximation of $f$ by $(K_{n}f)$ for\\  $n=10,20,30$ and $40$  }
		\label{fig:K10}
	\end{subfigure}
	\hfill
	\begin{subfigure}[t]{0.45\textwidth}
		\includegraphics[width=\linewidth]{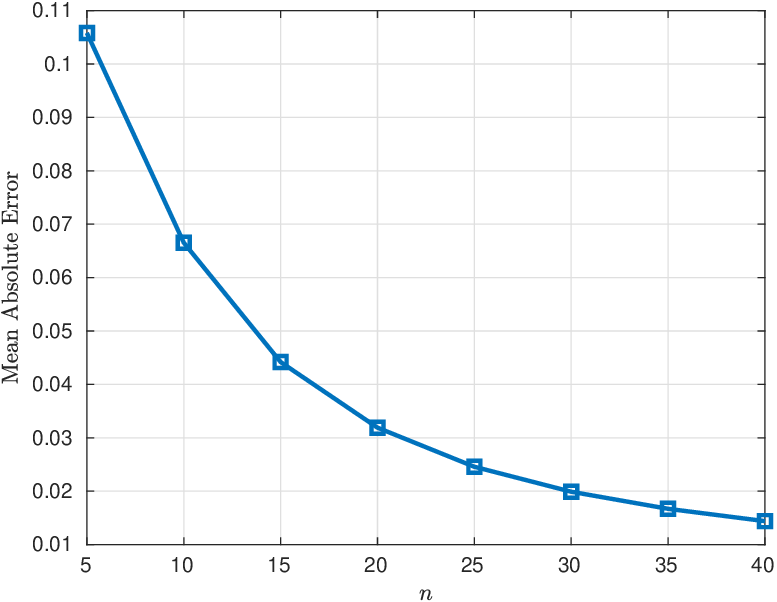} 
		\caption{Numerical evaluation of approximation error versus $n$}
		\label{fig:K20}
	\end{subfigure}
	\caption{Enhanced approximation of $f$ by $(K_{n}f)$  with increasing $n$ }
	\label{g2}
\end{figure}

			%
			%
		%

		\textbf{Example 2.} Let us consider the function $g : [0,1] \times [0,1] \rightarrow  \RR$ defined as $$g(u_1, u_2) = \frac{\sin\left(15 \cdot r(u_1, u_2)\right)}{1 + 10 \cdot r(u_1, u_2)},$$ where \( r(u_1, u_2) = \sqrt{(u_1 - 0.5)^2 + (u_2 - 0.5)^2} \). From Figures~\ref{g3} and~\ref{g4}, we observe that as \( n \) increases, the convergence behavior and approximation accuracy of the operator~\eqref{multi} improve significantly. Table~\ref{t2} and \ref{or2} presents the corresponding error-estimations for the function $g$ by $(K_{n}g)$ in the setting of mixed norm Lebesgue space and mixed norm Orlicz space.
		%
		
			\begin{figure}[htbp]
								\centering
								
								\begin{subfigure}{0.3\textwidth}
										\centering
										\includegraphics[width=\linewidth]{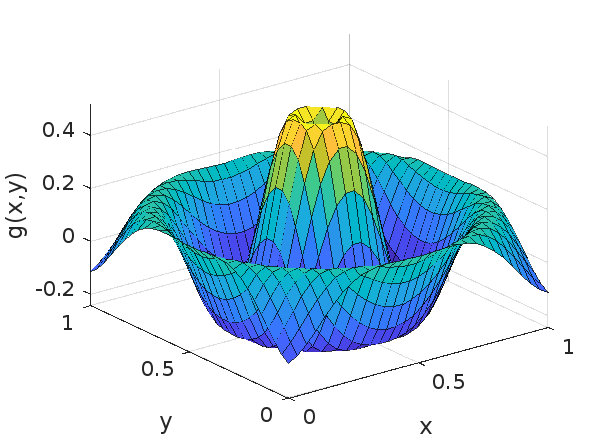}
										\caption{Original function }
									\end{subfigure}
								\hfill
								\begin{subfigure}{0.3\textwidth}
										\centering
										\includegraphics[width=\linewidth]{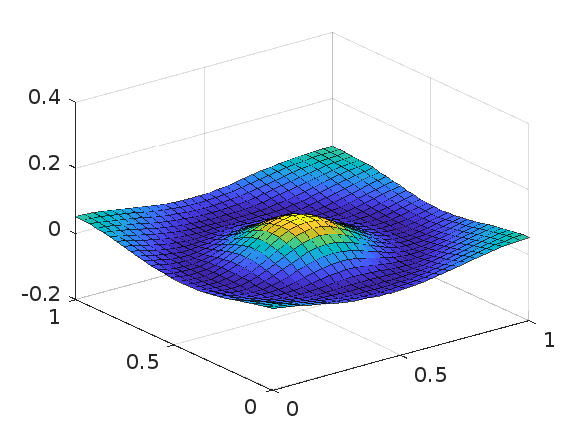}
										\caption{$K_{10} g$}
									\end{subfigure}
								\hfill
								\begin{subfigure}{0.3\textwidth}
										\centering
										\includegraphics[width=\linewidth]{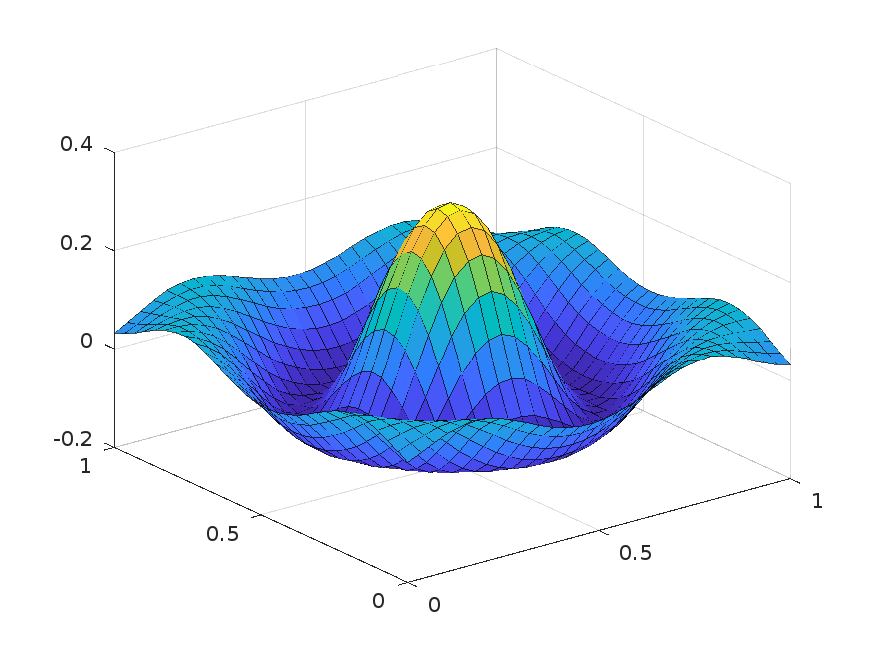}
										\caption{$K_{20} g$}
									\end{subfigure}
								
								\vspace{1em}
								
								\begin{subfigure}{0.3\textwidth}
										\centering
										\includegraphics[width=\linewidth]{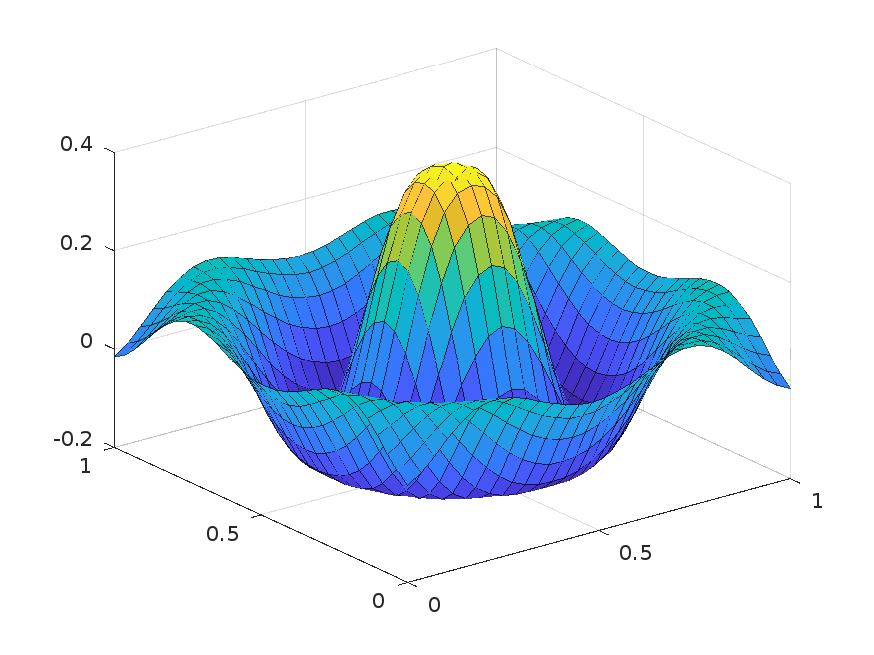}
										\caption{$K_{30} g$}
									\end{subfigure}
								\hfill
								\begin{subfigure}{0.3\textwidth}
										\centering
										\includegraphics[width=\linewidth]{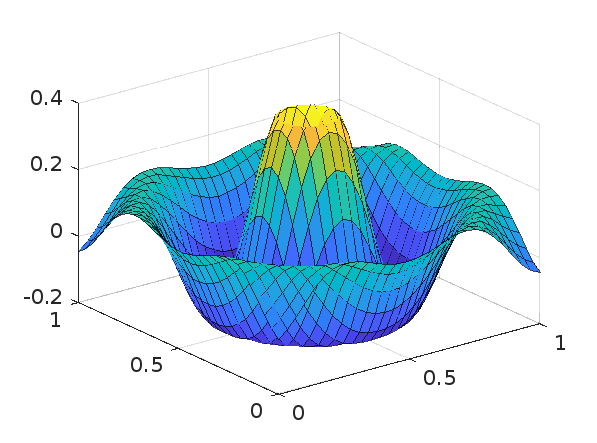}
										\caption{$K_{40} g$}
									\end{subfigure}
								\hfill
								\begin{subfigure}{0.3\textwidth}
										\centering
										\includegraphics[width=\linewidth]{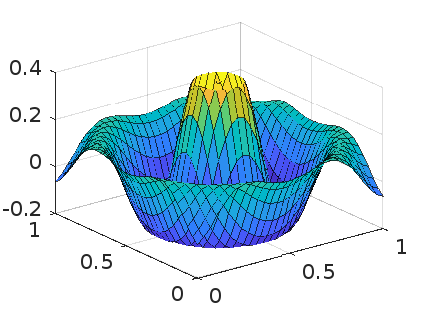}
										\caption{$K_{50} g$}
									\end{subfigure}
								
								\caption{Approximation of $g$ by $(K_{n} g)$ based on $\rho_l$ for different values of $n$}
								\label{g3}
							\end{figure}
		
			\begin{figure}[H]
			\centering
			\begin{subfigure}[t]{0.54\textwidth}
				\includegraphics[width=\linewidth]{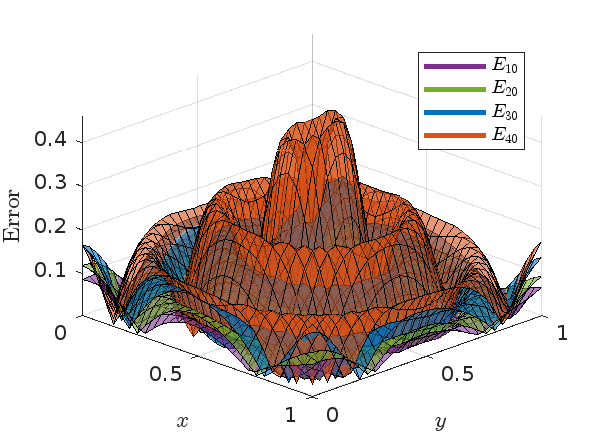}  
				\caption{Error of approximation of $g$ by $(K_{n}g)$ for\\  $n=10,20,30$ and $40$  }
				\label{fig:K10}
			\end{subfigure}
			\hfill
			\begin{subfigure}[t]{0.45\textwidth}
				\includegraphics[width=\linewidth]{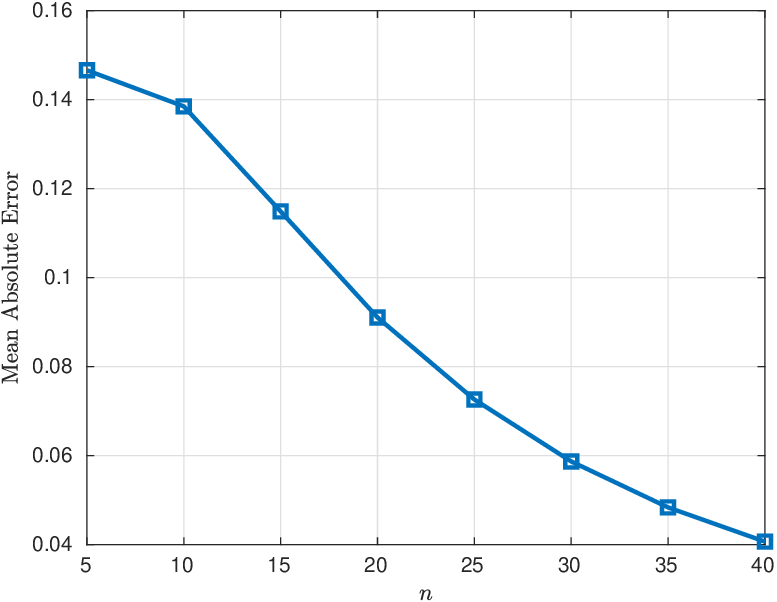}  
				\caption{Numerical evaluation of approximation error versus $n$}
				\label{fig:K20}
			\end{subfigure}
			\caption{Enhanced approximation of $g$ by $(K_{n}g)$  with increasing $n$ }
			\label{g4}
		\end{figure}
		
		\begin{table}[h!]
			\centering
			\renewcommand{\arraystretch}{1.2}
			\setlength{\tabcolsep}{10pt}
			\caption{Max absolute error and Mixed $L^{(p_1,p_2)}$-error of approximation of $(K_ng)$ to $g$  based on $\rho_l$}
			\label{t2}
			\vspace{0.5em}
			\rowcolors{2}{gray!10}{white}
			\begin{tabular}{>{\bfseries}c c c c c}
				\toprule
				\rowcolor{gray!25}
				$n$ & Max absolute error & Mixed $L^{3,4}$-error & Mixed $L^{2,5}$-error & Mixed $L^{3,6}$-error \\
				\midrule
				10 & 0.44481 & 0.20246 & 0.18735 & 0.21426 \\
				20 & 0.41317  & 0.13517 & 0.12613 & 0.14362 \\
				30 & 0.41143 &  0.096905 &  0.08971 & 0.10421\\
				40 & 0.32870  &  0.078261 & 0.070622 &  0.086156 \\
				\bottomrule
			\end{tabular}
		\end{table}

		\begin{table}[h!]
			\centering
			\renewcommand{\arraystretch}{1.3}
			\setlength{\tabcolsep}{12pt}
			\caption{Error-estimates for modular convergence of $g$ by $(K_ng)$.}
			\label{or2}
			\vspace{0.5em}
			\rowcolors{2}{gray!10}{white}
			\resizebox{0.75\textwidth}{!}{ 
				\begin{tabular}{c c c c c}
					\toprule
					\rowcolor{gray!25}
					$n$ 
					& \multicolumn{2}{c}{$L^{\left(\phi_{\alpha,\beta},\phi_\alpha\right)}$}
					& \multicolumn{2}{c}{$L^{\left(\phi_\alpha,\phi_{\alpha,\beta}\right)}$} 
					\\
					\cmidrule(lr){2-3} \cmidrule(lr){4-5}
					\rowcolor{gray!25}
					& \makecell{$\alpha = 1,$\\$\beta = 1.3$} 
					& \makecell{$\alpha = 2,$\\$\beta = 1.5$}
					& \makecell{$\alpha = 2,$\\$\beta = 1.2$} 
					& \makecell{$\alpha = 1.5,$\\$\beta = 2$} \\
					\midrule
					10  & 0.163866 & 0.001470 & 0.001239& 0.018829\\
					20 & 0.106316 & 0.000274 & 0.000237 & 0.007283 \\
					30 & 0.071351 & 0.000066 & 0.000058 & 0.003133 \\
					40 & 0.052765 & 0.000024& 0.000022 &  0.001679 \\
					\bottomrule
				\end{tabular}
			}
		\end{table}

			To validate Theorem \ref{thm 3.3} and \ref{interpol}, i.e., the convergence of multivariate Kantorovich-type NN operators $(K_nf)$ in the setting of mixed norm Lebesgue space $L^{\overrightarrow{\mathscr{P}}}(\mathcal{J})$ and mixed norm Orlicz space $L^{\overrightarrow{\Phi}}(\mathcal{J})$, we provide Tables \ref{t1}- \ref{or2} for error-estimation with respect to these function spaces.

			In Algorithm~\ref{alg1}, we provide the procedure for surface reconstruction in the context of bivariate function approximation using Kantorovich-type  neural network operators (\ref{multi}).

			\begin{algorithm}[H]
				\caption{Surface Reconstruction by Kantorovich-Type Neural Network Operators}
				\label{alg1}
				\begin{flushleft}
					\textbf{function} \texttt{Kantorovich-Approx}$(n, \rho(x), f(u_1, u_2))$\\
					\textbf{Input:} Parameter $n$, sigmoidal function $\rho(x)$, bivariate function $f(u_1, u_2)$\\
					\textbf{Output:} Reconstructed surface $\tilde{f}(x, y)$
				\end{flushleft}
				\vspace{-0.5em}
				\begin{align*}
					&\text{1. Define activation kernel: } K(x) = \frac{1}{2}(\rho(x+1) - \rho(x-1)) \\
					&\text{2. Define bivariate basis: } \Psi(x, y) = K(x) \cdot K(y) \\
					&\text{3. For } k_1, k_2 = -n \text{ to } n{-}1: \\
					&\quad I_{k_1, k_2} \leftarrow \left[\frac{k_1}{n}, \frac{k_1+1}{n}\right] \times \left[\frac{k_2}{n}, \frac{k_2+1}{n}\right] \\
					&\quad A_{k_1, k_2} \leftarrow n^2 \cdot \iint_{I_{k_1,k_2}} f(u_1, u_2) \, du_1 \, du_2 \\
					&\text{4. For each evaluation point } (x, y): \\
					&\quad \tilde{f}(x, y) = 
					\frac{
						\sum\limits_{k_1, k_2} A_{k_1, k_2} \cdot \Psi(n x - k_1, n y - k_2)
					}{
						\sum\limits_{k_1, k_2} \Psi(n x - k_1, n y - k_2)
					}
				\end{align*}
			\end{algorithm}


				\subsection{Applications in Digital Image Processing} \label{appimg}
				The family of NN operators has widespread applications in image processing.
				These operators serve as convenient tool for filling in missing information within a data set.
				In this process, an analog image is a function of two continuous variables, while a digital image is a discrete signal represented as a matrix. Throughout this procedure, the function corresponding to the original image matrix is used to obtain a new matrix using the multivariate Kantorovich-type NN operator serving as an approximator of the original image.
				An image with a specific resolution size $u \times v$ is a discrete structure composed of a finite set of pixels captured by an image system, from which a corresponding grayscale image matrix $C=(c_{ij})_{i,j \in\NN}, i=1,...,u; j=1,...,v$, \hspace{1pt} can be derived. The matrix representation of gray scale image matrix, denoted by $C$ can be viewed as a two-dimensional step function $I$ in $L^{(p_1,p_2)}(\RR^2)$, where $1 \leq p_1,p_2 <\infty.$ The function $I$ is defined as follows
				\begin{equation}\label{ii}
		I(x,y)=\sum_{i=1}^{u}\sum_{j=1}^{v}c_{ij}\cdot \mathbf{1}_{ij}(x,y)
				\end{equation}
				where 	\[
				\mathbf{1}_{ij}(x,y) := 
				\begin{cases} 
					1, &  (x,y)\in (i-1,i]\times (j-1,j], \\
					0, & otherwise.
				\end{cases}
				\]
				Thus, $I(x,y)$ maps every index pair $(i,j)$ to the corresponding value $c_{ij}$.
				In this study, two quantitative metrics: PSNR and SSIM are utilized to assess the quality of the reconstructed images.
				
				\textbf{The Structural Similarity Index Measure (SSIM)}: It evaluates the quality of an image primarily by analyzing its brightness $\mathcal{B}(A,B)$, contrast $\mathcal{C}(A,B)$, and structure $\mathcal{E}(A,B)$, defined by
				\[
				\mathcal{B}(A,B) = \frac{2\mu_A\mu_B + d_1}{\mu_A^2 + \mu_B^2 + d_1},\quad
				\mathcal{C}(A,B) = \frac{2\sigma_A \sigma_B + d_2}{\sigma_A^2 + \sigma_B^2 + d_2},
				\quad
				\mathcal{E}(A,B)= \frac{\sigma_{AB} + d_3}{\sigma_A \sigma_B + d_3}
				\] 
				where  $\mu_A$ and $\mu_B$ are the mean values,
				$\sigma_A^2$ and $\sigma_B^2$ are their variances, and
				$\sigma_{AB}$ is the covariance of $A$ and $B$. Moreover,
				$d_1 = (k_1L)^2$, $d_2 = (k_2L)^2$, $d_3 = \frac{d_2}{2}$, and
				$L$ is the range of pixel values with $k_1 = 0.01$, $k_2 = 0.03$, $k_3 = 0.03$.
				Then, the SSIM is expressed as
				\[
				SSIM(A,B) := [\mathcal{B}(A,B)^\alpha \cdot \mathcal{C}(A,B)^\beta \cdot \mathcal{E}(A,B)^\gamma].
				\]
				
				For $\alpha = \beta = \gamma = 1$, SSIM is expressed as
				\[
				SSIM(A,B) = \frac{(2\mu_A\mu_B + d_1)(2\sigma_{AB} + d_2)}{(\mu_A^2 + \mu_B^2 + d_1)(\sigma_A^2+ \sigma_B^2 + d_2)}.
				\]
				
				Next, we present the formulation of the Peak Signal to Noise Ratio (PSNR), which is defined in terms of the Mean Square Error (MSE).
				
				\textbf{The Peak Signal to Noise Ratio (PSNR):} The MSE of the reconstructed image
				can be written as
				\[
				MSE(A,B) = \frac{1}{XY}\sum_{i=1}^{X} \sum_{j=1}^{Y}[A(i,j)-B(i,j)]^2,
				\]
				where original and reconstructed images are denoted by \( A \) and \( B \) respectively. Similarly, \( A(i, j) \) and \( B(i, j) \) represent the pixel values at the corresponding coordinates. Moreover, the corresponding PSNR is given by
				$$PSNR = 10 \log_{10} \left( \frac{(MAX)^2}{MSE} \right),$$
				where \( `MAX$'$ \) denotes the maximum possible pixel value of the image.
				
				%
				%
				%
				
				\subsubsection{Image reconstruction}
				
				We begin by demonstrating the application of   (\ref{multi}) in image reconstruction. Using the function representation of an image 
			$I$, as described in (\ref{ii}), we apply Algorithm \ref{alg2} to reconstruct a grayscale image of a \textit{Bird 1} based on its corresponding image function and the two-dimensional version of the density function.

			\begin{figure}[h!]
				\centering
				
				\begin{subfigure}[c]{0.35\textwidth}
					\centering
					\includegraphics[width=\linewidth]{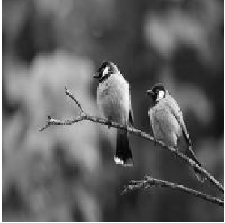}
					\vspace{-1em}
					\caption{Original image $128 \times 128$\\\makebox[\linewidth]{(Bird 1)}}
				\end{subfigure}
				\hfill
				\begin{subfigure}[c]{0.35\textwidth}
					\centering
					\includegraphics[width=\linewidth]{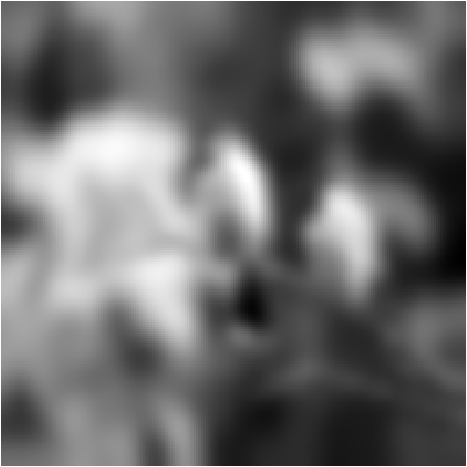}
					\vspace{-1em}
					\captionsetup{justification=centering, font=small, skip=2pt}
					\caption{Reconstructed image with  $n=50$ PSNR=22.64, SSIM=0.7625}
				\end{subfigure}
				
				\vspace{2mm}
				
				\begin{subfigure}[c]{0.35\textwidth}
					\centering
					\includegraphics[width=\linewidth]{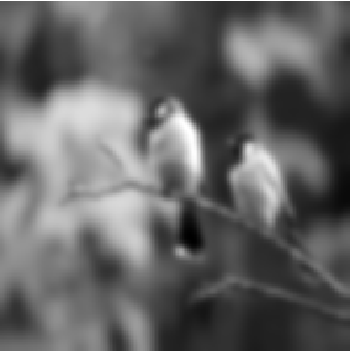}
					\vspace{-1em}
					\captionsetup{justification=centering, font=small, skip=2pt}
					\caption{Reconstructed image with $n=100$ PSNR=25.26, SSIM=0.8505}
				\end{subfigure}
				\hfill
				\begin{subfigure}[c]{0.35\textwidth}
					\centering
					\includegraphics[width=\linewidth]{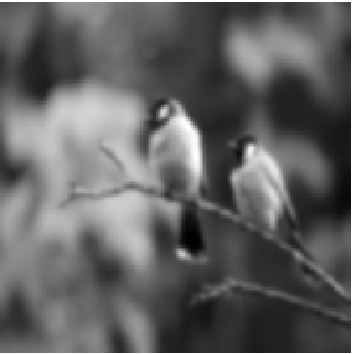}
					\vspace{-1em}
					\captionsetup{justification=centering, font=small, skip=2pt}
					\caption{Reconstructed image with $n=150$ PSNR=27.02, SSIM=0.8969}
				\end{subfigure}
				
				\vspace{2mm}
				
				\begin{subfigure}[c]{0.35\textwidth}
					\centering
					\includegraphics[width=\linewidth]{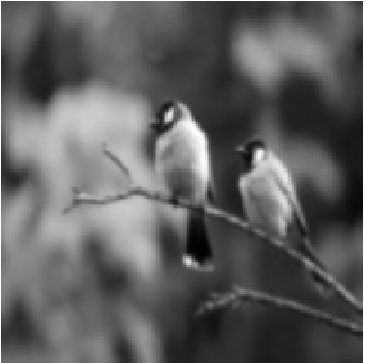}
					\vspace{-1em}
					\captionsetup{justification=centering, font=small, skip=2pt}
					\caption{Reconstructed image with $n=200$ PSNR=27.94, SSIM=0.9289}
				\end{subfigure}
				\hfill
				\begin{subfigure}[c]{0.35\textwidth}
					\centering
					\includegraphics[width=\linewidth]{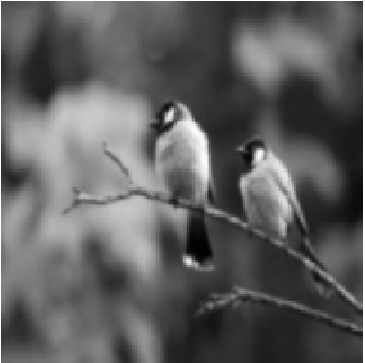}
					\vspace{-1em}
					\captionsetup{justification=centering, font=small, skip=2pt}
					\caption{Reconstructed image with  $n=250$ PSNR=29.36, SSIM=0.9466}
				\end{subfigure}
				
				\caption{Image reconstruction using multivariate Kantorovich-type  NN operators with $\rho_h$ for different values of $n$.}
				\label{g5}
			\end{figure}
			
				In Figure~\ref{g5}, we illustrate the visual quality of the reconstructed image for various randomly chosen values of \( n \). Along with this, we provide the performance evaluation of \(( K_{n}f) \) based on PSNR and SSIM metrics (see Table \ref{t11}). The procedure for reconstruction of grayscale image using multivariate Kantorovich-type neural network operators is described in Algorithm~\ref{alg2}.

			\begin{table}[h!]
				\centering
				\renewcommand{\arraystretch}{1.2}
				\setlength{\tabcolsep}{10pt}
				\caption{Evaluation of image reconstruction using PSNR and SSIM for different values of 
					$n$ based on multivariate Kantorovich-type  NN operators.}
				\label{t11}
				\vspace{0.5em}
				\rowcolors{2}{gray!10}{white}
				\begin{tabular}{>{\bfseries}c c c c c c}
					\toprule
					\rowcolor{gray!25}
					$n$ &  50 & 100 & 150 & 200 & 250 \\
					\midrule
					PSNR (dB)  & 22.64 & 25.26 & 27.02 & 27.94 & 29.36 \\
					SSIM       & 0.7625 & 0.8505 & 0.8969 & 0.9289 & 0.9466 \\
					\bottomrule
				\end{tabular}
			\end{table}

			\begin{algorithm}[H]
				\caption{Image reconstruction using multivariate Kantorovich-type  NN Operators  activated by $\rho_h$}
				\label{alg2}
				\begin{flushleft}
					\textbf{function} \texttt{Kantorovich-Reconstruct}$(n, \rho(x), I)$\\
					\textbf{Input:} Parameter $n$, sigmoidal function $\rho(x)$, grayscale image $I$\\
					\textbf{Output:} Reconstructed image $\tilde{I}$
				\end{flushleft}
				\begin{align*}
					&\text{1. Normalize input image } I \in [0,1] \text{ and convert to grayscale if needed} \\
					&\text{2. Let } (h_1, h_2) \gets \text{size}(I) \text{ and define normalized coordinates:} \\
					&\quad x = \frac{j - 1}{h_2 - 1},\quad y = \frac{i - 1}{h_1 - 1} \\
					&\text{3. Define kernel: } P(t) = \frac{1}{2} \left( \rho(t + 1) - \rho(t - 1) \right) \\
					&\text{4. Define continuous image access: } \\
					&\quad F(x, y) = I[\min(\lfloor y(h_1 - 1) \rfloor + 1,\ h_1),\ \min(\lfloor x(h_2 - 1) \rfloor + 1,\ h_2)] \\
					&\text{5. For each } (x, y) \in [0,1]^2 \text{ on the image domain:} \\
					&\quad \tilde{I}(x,y) = \frac{ \sum\limits_{k_1, k_2} K_n(F)(k_1, k_2) \cdot P(n x - k_1) \cdot P(n y - k_2) }{ \sum\limits_{k_1, k_2} P(n x - k_1) \cdot P(n y - k_2) } \\
					&\quad \text{where } K_n(F)(k_1, k_2) = \text{Kantorovich-type average of } F \text{ over a small neighborhood} \\
					&\text{6. Normalize } \tilde{I} \text{ to } [0,1] \text{ for display} \\
					&\text{7. Compute error metrics: PSNR}(I, \tilde{I}),\ \text{SSIM}(I, \tilde{I})
				\end{align*}
			\end{algorithm}

					\subsubsection{Image inpainting}
					
%
In this subsection, we discuss the application of the operator~\eqref{multi} in image inpainting, which is an essential technique in image processing that aims to restore missing or damaged regions of an image, ensuring both visual and structural consistency.
 Image inpainting is widely used in many applications to remove unwanted objects by filling in corrupted or missing regions with estimated pixel values.
  This method is capable of addressing various types of distortions, including text overlays, noise, block artifacts, scratches, lines, and different types of masks.

\clearpage
To evaluate the performance of our inpainting approach, we consider the gray-scale image \textit{Bird 2}, as shown in Figure~\ref{bird2}. We simulate missing data by introducing artificial masks randomly across the image, resulting in approximately 21\% of the pixels being removed, as illustrated in Figure~\ref{Masked image}. 
For each masked pixel at location $(x_1, x_2)$, Algorithm \ref{alg3} is employed to estimate its intensity by utilizing the intensities of neighboring unmasked pixels.
Since the original pixel values are available, we quantitatively assess the effectiveness of the inpainting using standard image quality metrics—PSNR and SSIM—by comparing the inpainted image and the original image (see Table \ref{t22}).

\begin{table}[h!]
	\centering
	\renewcommand{\arraystretch}{1.2}
	\setlength{\tabcolsep}{10pt}
	\caption{Evaluation of image inpainting  quality using PSNR and SSIM for different values of 
		$n$ based on multivariate Kantorovich-type  NN operators.}
	\label{t22}
	\vspace{0.5em}
	\rowcolors{2}{gray!10}{white}
	\begin{tabular}{>{\bfseries}c c c c c}
		\toprule
		\rowcolor{gray!25}
		$n$ & 10 & 50  & 100 & 150  \\
		\midrule
		PSNR (dB) & 27.72  & 28.72 & 30.05 & 31.21  \\
		SSIM      & 0.8767 & 0.9118 & 0.9380 & 0.9552 \\
		\bottomrule
	\end{tabular}
\end{table}

\begin{algorithm}[H]
	\caption{Image inpainting using multivariate Kantorovich-type  NN Operators  activated by $\rho_l$}
	\label{alg3}
	\begin{flushleft}
		\textbf{function} \texttt{Image-Kantorovich-Inpaint}$(n, m, \rho(x), I, a_1, b_1, a_2, b_2)$\\
		\textbf{Input:} Parameter $n$,  $m$, sigmoidal function $\rho(x)$, grayscale image $I$, $a_1, b_1, a_2, b_2 \in [0,1]$\\
		\textbf{Output:} Inpainted image $\tilde{I}$
	\end{flushleft}
	\begin{align*}
		&\text{1. Normalize input image } I \in [0,1] \text{ and convert to grayscale if needed} \\
		&\text{2. Define image size: } (H, W) \gets \text{size}(I) \\
		&\text{3. Generate binary mask } M \text{ with percentage } p \text{ of missing pixels} \\
		&\text{4. Create masked image: } I_{\text{masked}}(i,j) =
		\begin{cases}
			I(i,j), & \text{if } M(i,j) = 1 \\
			\text{NaN}, & \text{otherwise}
		\end{cases} \\
		&\text{5. Define kernel: } P(x) = \frac{1}{2}(\rho(x+1) - \rho(x-1)) \\
		&\text{6. Define access function: } F(x,y) = \texttt{get\_image\_value}(I_{\text{masked}}, x, y) \\
		&\text{7. For each pixel } (i,j) \in [1, H] \times [1, W]: \\
		&\quad \text{If } I_{\text{masked}}(i,j) = \text{NaN:} \\
		&\quad \quad x_1 = \frac{j - 1}{W - 1},\quad x_2 = \frac{i - 1}{H - 1} \\
		&\quad \quad \tilde{I}(i,j) = \frac{
			\sum\limits_{k_1,k_2} A_{k_1,k_2} \cdot P(n x_1 - k_1) \cdot P(n x_2 - k_2)
		}{
			\sum\limits_{k_1,k_2} P(n x_1 - k_1) \cdot P(n x_2 - k_2)
		} \\
		&\quad \quad \text{where } A_{k_1,k_2} = n^2 \cdot \sum\limits_{p,q=1}^m F\left(\frac{k_1}{n} + \frac{p{-}0.5}{n m},\ \frac{k_2}{n} + \frac{q{-}0.5}{n m} \right) \cdot \frac{1}{m^2} \\
		&\quad \text{Else: } \tilde{I}(i,j) = I_{\text{masked}}(i,j) \\
		&\text{8. Normalize } \tilde{I} \text{ to } [0,1] \text{ for display}
	\end{align*}
\end{algorithm}

\begin{figure}[h!]
	\centering
	
	\begin{subfigure}[c]{0.35\textwidth}
		\centering
		\includegraphics[width=\linewidth]{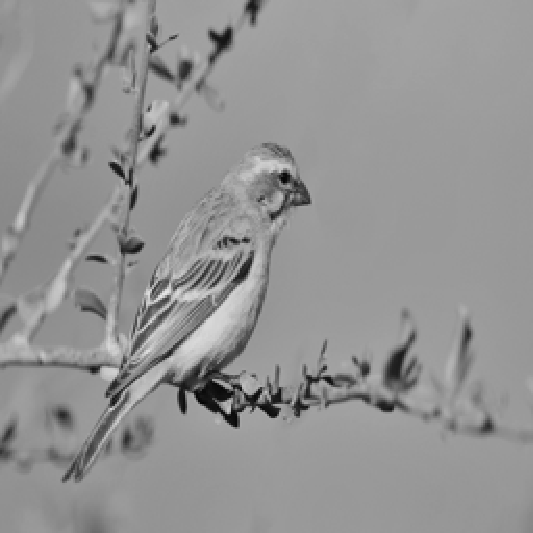}
		\vspace{-1em}
		\caption{Original image $128 \times 128$\\\makebox[\linewidth]{(Bird 2)}}
		\label{bird2}
	\end{subfigure}
	\hfill
	\begin{subfigure}[c]{0.35\textwidth}
		\centering
		\includegraphics[width=\linewidth]{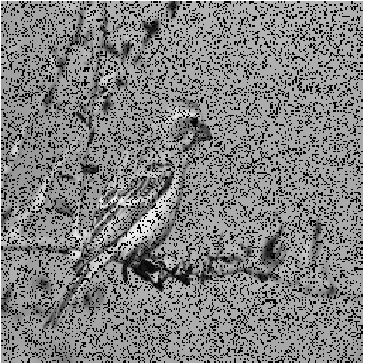}
		\vspace{-1em}
		\captionsetup{justification=centering, font=small, skip=2pt}
		\caption{Masked image (21\% missing pixels )}
		\label{Masked image}
	\end{subfigure}
	
	\vspace{2mm}
	
	\begin{subfigure}[c]{0.35\textwidth}
		\centering
		\includegraphics[width=\linewidth]{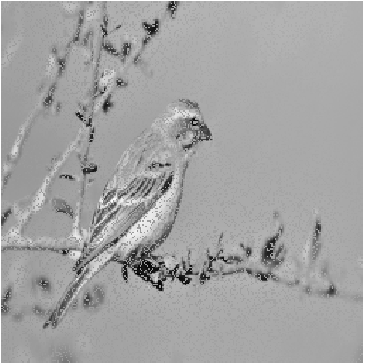}
		\vspace{-1em}
		\captionsetup{justification=centering, font=small, skip=2pt}
		\caption{ Inpainting image with $n=10$ PSNR=27.72, SSIM=0.8767}
	\end{subfigure}
	\hfill
	\begin{subfigure}[c]{0.35\textwidth}
		\centering
		\includegraphics[width=\linewidth]{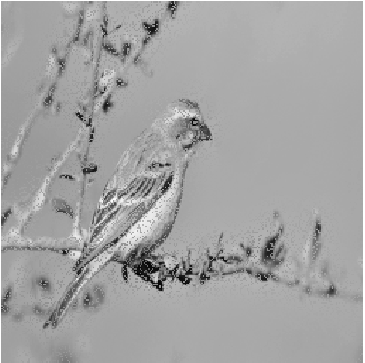}
		\vspace{-1em}
		\captionsetup{justification=centering, font=small, skip=2pt}
		\caption{Inpainting image with $n=50$ PSNR=28.72, SSIM=0.9118}
	\end{subfigure}
	
	\vspace{2mm}
	
	\begin{subfigure}[c]{0.35\textwidth}
		\centering
		\includegraphics[width=\linewidth]{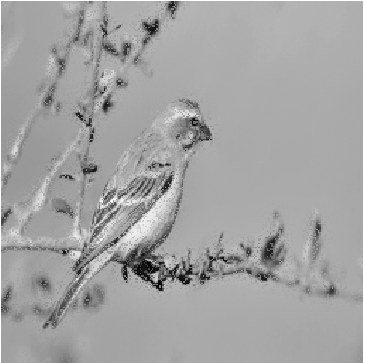}
		\vspace{-1em}
		\captionsetup{justification=centering, font=small, skip=2pt}
		\caption{Inpainting image with $n=100$ PSNR=30.05, SSIM=0.9380}
	\end{subfigure}
	\hfill
	\begin{subfigure}[c]{0.35\textwidth}
		\centering
		\includegraphics[width=\linewidth]{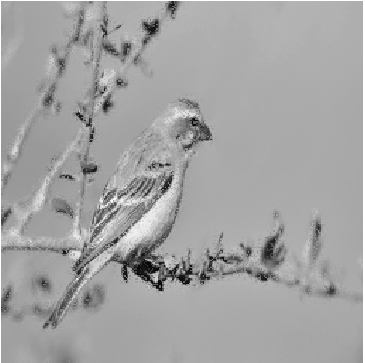}
		\vspace{-1em}
		\captionsetup{justification=centering, font=small, skip=2pt}
		\caption{Inpainting image with  $n=150$ PSNR=31.21, SSIM=0.9552}
	\end{subfigure}
	
	\caption{Image inpainting using multivariate Kantorovich-type  NN operator with $\rho_l$ for different values of $n$.}
	\label{g6}
\end{figure}

					\subsubsection{Image scaling}

					Here, we aim to study the image scaling capabilities of the operator defined in~(\ref{multi}).
							Image scaling is a fundamental tool in image processing used to resize images either by enlarging (upscaling) or reducing (downscaling) their dimensions. It plays a crucial role in various applications such as image enhancement, compression, and multi-resolution analysis. Traditionally, scaling was applied to adapt images to different display resolutions or storage formats.
							
								In this experiment, we first downsample the grayscale image by a factor of \( \frac{1}{2} \), and then apply an upscaling procedure using the Kantorovich-type NN operator with a magnification factor of 2.
							The quality of the reconstructed (upscaled) images is quantitatively assessed using standard image quality metrics, namely PSNR and SSIM (see Table \ref{t5}). A detailed procedure for grayscale image scaling using multivariate Kantorovich-type neural network operators is provided in Algorithm~\ref{alg4}.
						
						
							\begin{figure}[h!]
							\centering
							
							\begin{minipage}{0.5\textwidth}
								\centering
								\includegraphics[width=\linewidth]{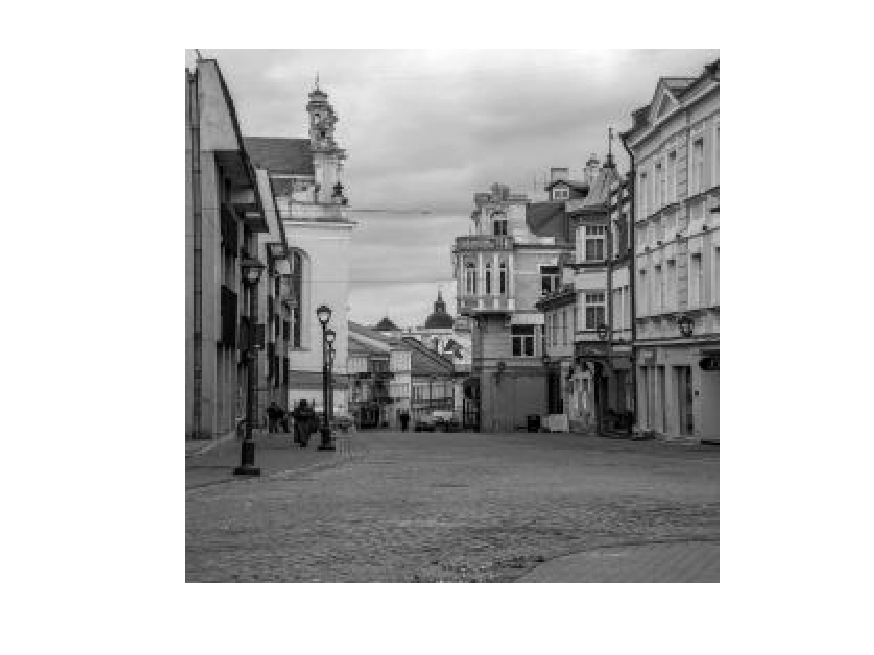}
								\vspace{-2.5em}
								\caption*{Original Image}
							\end{minipage}
							
							\vspace{4mm}  
							
							\begin{minipage}{0.49\textwidth}
								\centering
								\includegraphics[width=\linewidth]{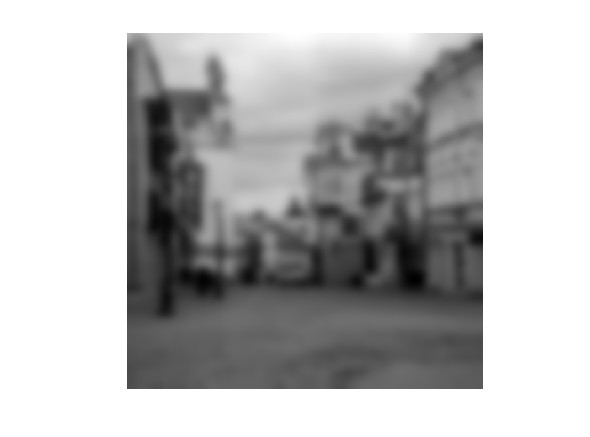}
								\vspace{-2.5em}
								\caption*{Upscaled image for $n=80$}
							\end{minipage}
							\hfill
							\begin{minipage}{0.49\textwidth}
								\centering
								\includegraphics[width=\linewidth]{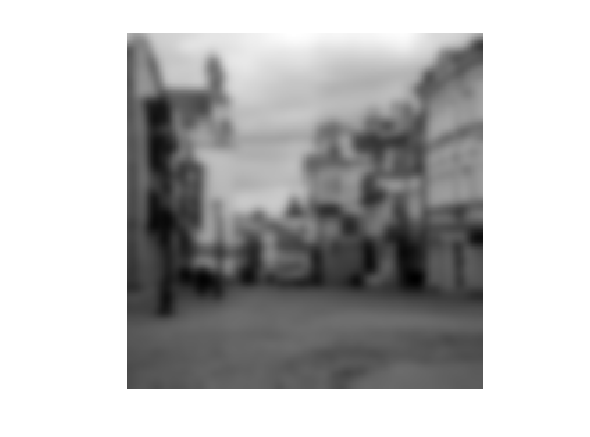}
								\vspace{-2.5em}
								\caption*{Downscaled image for $n=80$ }
							\end{minipage}
							
							\vspace{3mm}
							
							\begin{minipage}{0.49\textwidth}
								\centering
								\includegraphics[width=\linewidth]{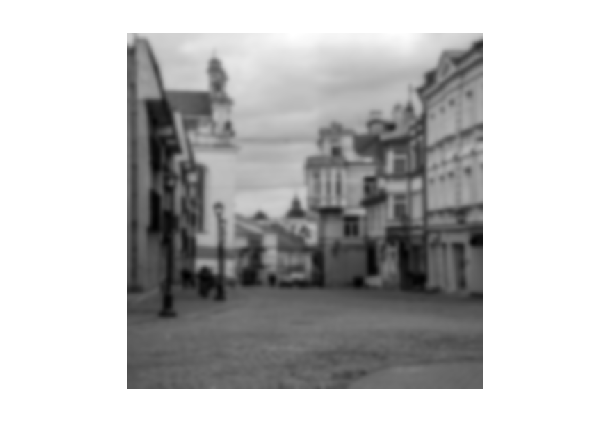}
								\vspace{-2.5em}
								\caption*{Upscaled image for $n=180$}
							\end{minipage}
							\hfill
							\begin{minipage}{0.49\textwidth}
								\centering
								\includegraphics[width=\linewidth]{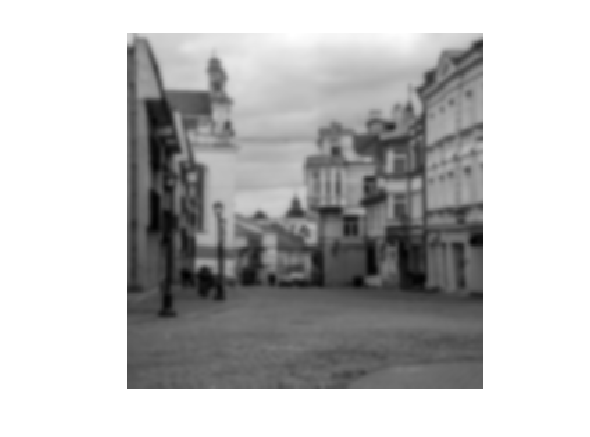}
								\vspace{-2.5em}
								\caption*{Downscaled image for $n=180$}
							\end{minipage}
							
							\vspace{2mm}
							\caption{Image scaling using multivariate Kantorovich-type NN operator with $\rho_h$ for different values of $n$.}
							\label{g7}
						\end{figure}

				\clearpage
			
					\begin{algorithm}[H]
						\caption{Image scaling using multivariate Kantorovich-type NN Operators  activated by $\rho_h$}
						\label{alg4}
						\begin{flushleft}
							\textbf{function} \texttt{Image-Kantorovich-Reconstruct}$(n, m, \rho(x), I)$\\
							\textbf{Input:} Parameters $n$, $m$, sigmoidal function $\rho(x)$, grayscale image $I$\\
							\textbf{Output:} Upsampled image $\tilde{I}$, downsampled image $D$
						\end{flushleft}
						\begin{align*}
							&\text{1. Normalize input image } I \in [0,1] \text{ and convert to grayscale if needed} \\
							&\text{2. Define image size: } (h_1, h_2) \gets \text{size}(I),\quad S \gets \text{upsampling factor} \\
							&\text{3. Define kernel: } P(x) = \frac{1}{2}(\rho(x+1) - \rho(x-1)) \\
							&\text{4. Define continuous access: } F(x,y) = I[\text{round}(y \cdot (h_1{-}1)) + 1,\ \text{round}(x \cdot (h_2{-}1)) + 1] \\
							&\text{5. For each pixel } (x, y) \in [0,1]^2 \text{ on the upsampled grid:} \\
							&\quad \tilde{I}(x,y) = \frac{ \sum\limits_{k_1,k_2} A_{k_1,k_2} \cdot P(n x - k_1) \cdot P(n y - k_2) }{ \sum\limits_{k_1,k_2} P(n x - k_1) \cdot P(n y - k_2) } \\
							&\quad \text{where } A_{k_1,k_2} = n^2 \cdot \sum\limits_{p,q=1}^m F\left(\frac{k_1}{n} + \frac{p{-}0.5}{n m},\ \frac{k_2}{n} + \frac{q{-}0.5}{n m} \right) \cdot \frac{1}{m^2} \\
							&\text{6. Normalize } \tilde{I} \text{ to } [0,1] \text{ for display} \\
							&\text{7. Downsample image: } D = \tilde{I}[S{:}S{:}end,\ S{:}S{:}end] \\
							&\text{8. Compute error metrics: PSNR}(I, D),\ \text{SSIM}(I, D)
						\end{align*}
					\end{algorithm}
					
						\begin{table}[h!]
						\centering
						\renewcommand{\arraystretch}{1.2}
						\setlength{\tabcolsep}{10pt}
						\caption{Evaluation of image scaling quality using PSNR and SSIM for different values of 
							$n$ based on multivariate Kantorovich-type  NN operators.}
						\label{t5}
						\vspace{0.5em}
						\rowcolors{2}{gray!10}{white}
						\begin{tabular}{>{\bfseries}c c c c c c c}
							\toprule
							\rowcolor{gray!25}
							$n$ & 20 & 50 & 80 & 110 & 150 & 180 \\
							\midrule
							PSNR (dB) & 19.04 & 20.53 & 21.46 & 22.24 & 23.06 & 23.56 \\
							SSIM      & 0.4180 & 0.4471 & 0.4907 & 0.5419 & 0.6012 & 0.6396 \\
							\bottomrule
						\end{tabular}
					\end{table}

					\subsubsection{Image denoising}
					We now demonstrate the image denoising performance of the operator~\eqref{multi}.
					In this process, we aim to recover the underlying clean image by effectively removing noise from a corrupted observation. A major challenge arises from the fact that noise, edges, and fine textures all reside in the high-frequency domain, making them difficult to separate during the denoising process. Noise in images can negatively affect later processing tasks such as video analysis, image recognition, and tracking. Therefore, removing noise, or image denoising, is an important step in modern image processing to improve image quality and support further analysis.
					Speckle noise is a type of grainy distortion commonly seen in medical and radar images. It is caused by the interference of reflected signals and reduces image clarity by masking fine details. Removing speckle noise is important for improving image quality and accurate analysis.  It is evident that the increasing value of $n$ reduces the impact of noise in the image (see Table \ref{t33}).
					Algorithm~\ref{alg5} outlines the step-by-step process for grayscale image denoising based on multivariate Kantorovich-type neural network operators.
					

						\begin{figure}[h!]
						\centering
						
						\begin{subfigure}[c]{0.35\textwidth}
							\centering
							\includegraphics[width=\linewidth]{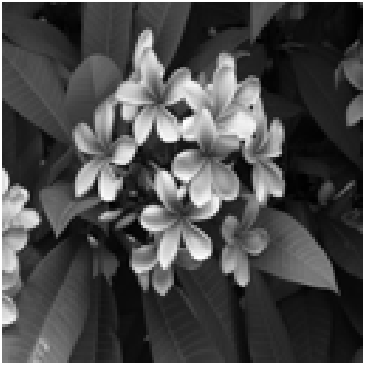}
							\vspace{-1em}
							\caption{Original image }
						\end{subfigure}
						\hfill
						\begin{subfigure}[c]{0.35\textwidth}
							\centering
							\includegraphics[width=\linewidth]{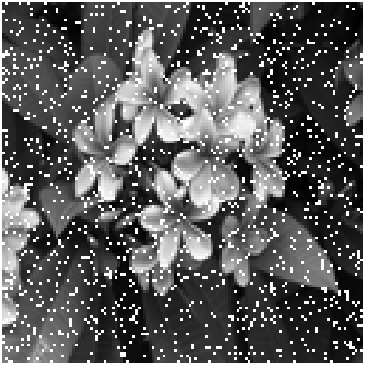}
							\vspace{-1em}
							\captionsetup{justification=centering, font=small, skip=2pt}
							\caption{Noisy image (Sparkle noise)}
						\end{subfigure}
						
						\vspace{2mm}
						
						\begin{subfigure}[c]{0.35\textwidth}
							\centering
							\includegraphics[width=\linewidth]{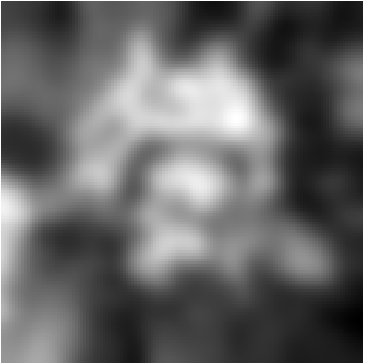}
							\vspace{-1em}
							\captionsetup{justification=centering, font=small, skip=2pt}
							\caption{Denoised image with $n=50$ PSNR=16.03, SSIM=0.3646}
						\end{subfigure}
						\hfill
						\begin{subfigure}[c]{0.35\textwidth}
							\centering
							\includegraphics[width=\linewidth]{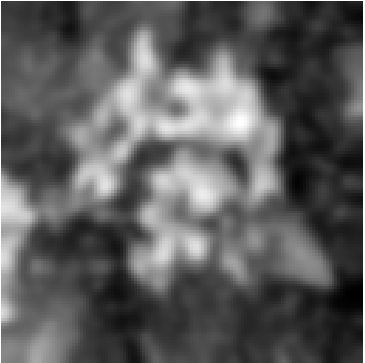}
							\vspace{-1em}
							\captionsetup{justification=centering, font=small, skip=2pt}
							\caption{Denoised image with $n=100$ PSNR=17.78, SSIM=0.4927}
						\end{subfigure}
						
						\vspace{2mm}
						
						\begin{subfigure}[c]{0.35\textwidth}
							\centering
							\includegraphics[width=\linewidth]{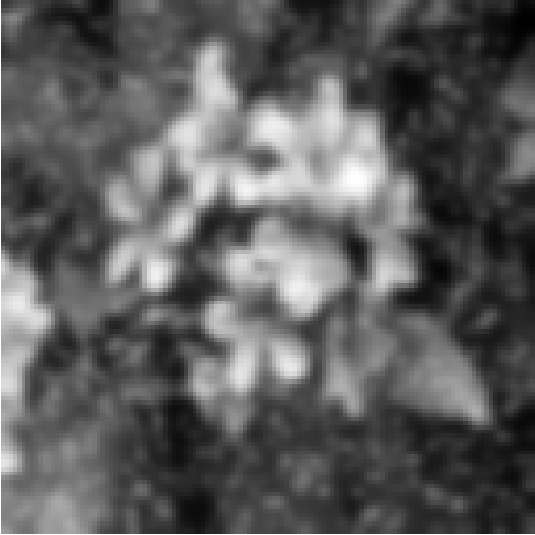}
							\vspace{-1em}
							\captionsetup{justification=centering, font=small, skip=2pt}
							\caption{Denoised image with $n=150$ PSNR=18.73, SSIM=0.5347}
						\end{subfigure}
						\hfill
						\begin{subfigure}[c]{0.35\textwidth}
							\centering
							\includegraphics[width=\linewidth]{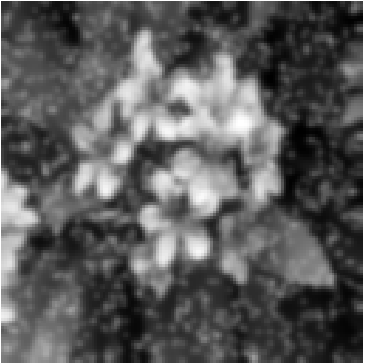}
							\vspace{-1em}
							\captionsetup{justification=centering, font=small, skip=2pt}
							\caption{Denoised image with  $n=200$ PSNR=19.03, SSIM=0.5584}
						\end{subfigure}
						
						\caption{Image denoising using multivariate Kantorovich-type NN operator with $\rho_l$ for different values of $n$.}
						\label{g8}
					\end{figure}

						\begin{table}[h!]
						\centering
						\renewcommand{\arraystretch}{1.2}
						\setlength{\tabcolsep}{10pt}
						\caption{Evaluation of image denoising quality using PSNR and SSIM for different values of 
							$n$ based on multivariate Kantorovich-type  NN operators.}
						\label{t33}
						\vspace{0.5em}
						\rowcolors{2}{gray!10}{white}
						\begin{tabular}{>{\bfseries}c c c c c}
							\toprule
							\rowcolor{gray!25}
							$n$ &  50 &  100 & 150 & 200 \\
							\midrule
							PSNR (dB) &  16.03 & 17.78 & 18.73 & 19.03 \\
							SSIM      &   0.3646 & 0.4927 & 0.5347 & 0.5584 \\
							\bottomrule
						\end{tabular}
					\end{table}

					\clearpage

					\begin{algorithm}[H]
						\caption{Image denoising using multivariate Kantorovich-type   NN Operators activated by $\rho_l$}
						\label{alg5}
						\begin{flushleft}
							\textbf{function} \texttt{Kantorovich-Denoise}$(n, m, \rho(x), I)$\\
							\textbf{Input:} Parameters $n$, $m$, sigmoidal function $\rho(x)$, grayscale image $I$\\
							\textbf{Output:} Denoised image $\tilde{I}$
						\end{flushleft}
						\begin{align*}
							&\text{1. Resize and normalize image } I \text{ to } [0,1] \text{ and size } 128 \times 128 \\
							&\text{2. Add sparkle noise (white impulse noise) with density } d \text{ to image} \\
							&\text{3. Define kernel: } P(x) = \frac{1}{2} \left( \rho(x + 1) - \rho(x - 1) \right) \\
							&\text{4. Define continuous access from noisy image:} \\
							&\quad F(x, y) = I_{\text{noisy}}\left[\min(\lfloor y(h_1 - 1) \rfloor + 1,\ h_1),\ \min(\lfloor x(h_2 - 1) \rfloor + 1,\ h_2)\right] \\
							&\text{5. For each pixel } (x, y) \in [0,1]^2 \text{ on the image grid:} \\
							&\quad \tilde{I}(x,y) = \frac{ \sum\limits_{k_1, k_2} A_{k_1,k_2} \cdot P(n x - k_1) \cdot P(n y - k_2) }{ \sum\limits_{k_1, k_2} P(n x - k_1) \cdot P(n y - k_2) } \\
							&\quad \text{where } A_{k_1,k_2} = n^2 \cdot \sum\limits_{p,q=1}^{m} F\left(\frac{k_1}{n} + \frac{p - 0.5}{n m},\ \frac{k_2}{n} + \frac{q - 0.5}{n m} \right) \cdot \frac{1}{m^2} \\
							&\text{6. Normalize } \tilde{I} \text{ to } [0,1] \text{ for visualization} \\
							&\text{7. Compute metrics: PSNR}(I, \tilde{I}),\ \text{SSIM}(I, \tilde{I})
						\end{align*}
					\end{algorithm}

						\subsection{Applications of Mixed Norm Structures}\label{appmix}

					Mixed norm spaces $L^{(p_1,p_2)}$ generalize the traditional \( L^{p} \) spaces and are particularly effective for analyzing multidimensional data such as signals and images. Unlike standard norms, which treat data uniformly, mixed norm apply different norms along different dimensions, enabling anisotropic analysis.
					To demonstrate this, we compare the standard \( L^p \) norm with a mixed \( L^{(p_1, p_2)} \) norm in quantifying reconstruction error for a synthetic 2D signal.


					\begin{figure}[h!]
					\centering
					\begin{subfigure}{0.3\textwidth}
						\includegraphics[width=\linewidth]{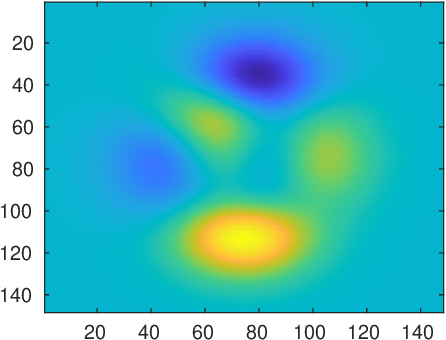}
						\captionsetup{justification=centering, font=small, skip=2pt}
						\caption{Clean signal (Original)}
					\end{subfigure}
					\hfill
					\begin{subfigure}{0.3\textwidth}
						\includegraphics[width=\linewidth]{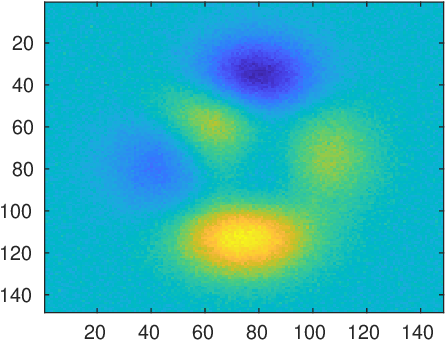}
						\captionsetup{justification=centering, font=small, skip=2pt}
						\caption{Noisy signal  }
					\end{subfigure}
					\hfill
					\begin{subfigure}{0.3\textwidth}
						\includegraphics[width=\linewidth]{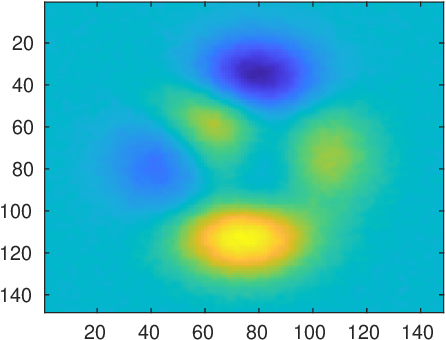}
						\captionsetup{justification=centering, font=small, skip=2pt}
						\caption{Reconstructed signal}
					\end{subfigure}
					\caption{Comparison of clean, noisy, and reconstructed (denoised) 2D signal}
					\label{m11}
				\end{figure}

					The steps followed in this experiment are listed below: 
					
					\begin{enumerate}
						\item A synthetic 2D signal is generated using \texttt{peaks} function with grid size \(148 \times 148\), which produces a smooth surface with multiple peaks and valleys.
						
						\item Gaussian noise with zero mean and standard deviation \(0.3\) is added element-wise to simulate measurement.
						
						\item The noisy signal is denoised using a Gaussian low-pass filter with standard deviation \(\sigma = 1\).
						
						\item Reconstruction errors are computed in two ways: standard $L^{(p_1,p_1)} $ and mixed $L^{(p_1,p_2)}$ norm.
						

					\end{enumerate}

					\begin{table}[h!]
						\centering
						\renewcommand{\arraystretch}{1.7}
						\setlength{\tabcolsep}{15pt}
						\caption{Comparison of $L^{(p_1,p_1)}$ and mixed $L^{(p_1,p_2)}$-error for denoised 2D signal}
						\label{m1}
						\vspace{0.7em}
						\rowcolors{2}{gray!10}{white}
						\begin{tabular}{>{\bfseries}c c c c c c}
							\toprule
							\rowcolor{gray!25}
							$p_1$ & {$L^{(p_1,p_1)}$-error} & $(p_1,p_2)$ & {$L^{(p_1,p_2)}$-error } & $(p_1,p_2)$ & {$L^{(p_1,p_2)}$-error } \\
							\midrule
							2 & 13.12861   & (2,3) & 5.74018& (2,4) & 3.81862     \\
							3 & 2.89823   & (3,4) & 1.92922 & (3,5) & 1.51235    \\
							4 & 1.41900     & (4,5) & 1.11907 & (4,6) & 0.95241   \\
							5 & 0.94815   & (5,6) & 0.81432  & (5,7) & 0.72688    \\
							6 & 0.73864   &  (6,7) &  0.66584 & (6,8)  & 0.61599  \\
							7 &  0.62975  & (7,8) & 0.58520  & (7,9) & 0.55145\\
							8 &  0.56583 &  (8,9) &  0.53762  & (8,10) &  0.51063 \\
							\bottomrule
						\end{tabular}
					\end{table}
					
					Figure~\ref{m1} shows the reconstruction of a 2D signal corrupted by Gaussian noise. The reconstruction quality is measured using the standard 
					$L^{p_1}$
					norm and the mixed $L^{(p_1,p_2)}$
					norm (see Table \ref{m1}).
					
					
					Next, we apply the same procedure to image data, similar to the 2D signal case described above. The key steps are outlined below:
					
					\begin{enumerate}
						\item A grayscale image is loaded and resized to \(256 \times 256\), then normalized to [0, 1].
						\item Salt and pepper noise with density 0.05 is added.
						\item The noisy image is denoised using a median filter.
						\item Reconstruction error is computed using: standard $L^{(p_1,p_1)} $ and mixed $L^{(p_1,p_2)}$ norm.
					\end{enumerate}

					\begin{figure}[h!]
						\centering
						\begin{subfigure}{0.3\textwidth}
							\includegraphics[width=\linewidth]{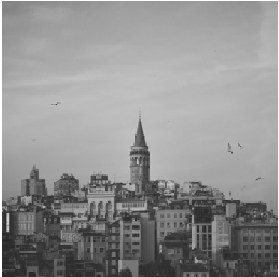}
							\captionsetup{justification=centering, font=small, skip=2pt}
							\caption{Clean  image (Original)}
						\end{subfigure}
						\hfill
						\begin{subfigure}{0.3\textwidth}
							\includegraphics[width=\linewidth]{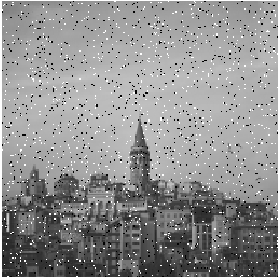}
							\captionsetup{justification=centering, font=small, skip=2pt}
							\caption{Noisy image}
						\end{subfigure}
						\hfill
						\begin{subfigure}{0.3\textwidth}
							\includegraphics[width=\linewidth]{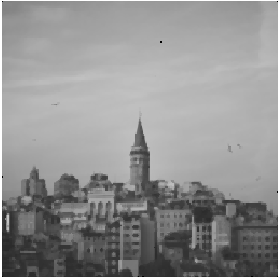}
							\captionsetup{justification=centering, font=small, skip=2pt}
							\caption{Reconstructed image}
						\end{subfigure}
						\caption{Comparison of clean, noisy, and reconstructed (denoised) image.}
						\label{m22}
					\end{figure}

					\begin{table}[h!]
						\centering
						\renewcommand{\arraystretch}{1.7}
						\setlength{\tabcolsep}{15pt}
						\caption{Comparison of $L^{(p_1,p_1)}$ and mixed $L^{(p_1,p_2)}$-error for denoised image}
						\label{m2}
						\vspace{0.7em}
						\rowcolors{2}{gray!10}{white}
						\begin{tabular}{>{\bfseries}c c c c c c}
							\toprule
							\rowcolor{gray!25}
							$p_1$ & {$L^{(p_1,p_1)}$-error} & $(p_1,p_2)$ & {$L^{(p_1,p_2)}$-error } & $(p_1,p_2)$ & {$L^{(p_1,p_2)}$-error } \\
							\midrule
							2 & 6.846742    & (2,3) & 3.165071 & (2,4) & 2.185745     \\
							3 &  1.720873    & (3,4) &  1.195194 & (3,5) & 0.976082    \\
							4 &  1.100493     & (4,5) & 0.930642 & (4,6) & 0.855222   \\
							5 & 0.954135   & (5,6) & 0.904091  & (5,7) & 0.881201    \\
							6 & 0.853499   &  (6,7) &   0.844664 & (6,8)  & 0.841901  \\
							7 &  0.886418  & (7,8) & 0.854099  & (7,9) &  0.830459\\
							8 &   0.769987  &  (8,9) &   0.761601  & (8,10) & 0.755460 \\
							\bottomrule
						\end{tabular}
					\end{table}
					
					Figure~\ref{m22} illustrates the reconstruction of an image affected by salt and pepper noise. We assess the reconstruction quality using both the standard 	$L^{p_1}$
					norm and the mixed $L^{(p_1,p_2)}$
					norm (see Table \ref{m2}). 

					\begin{table}[h!]
						\centering
						\renewcommand{\arraystretch}{1.7}
						\setlength{\tabcolsep}{9pt}
						\caption{Comparison of $L^{(p_1,p_1)}$ and mixed $L^{(p_1,p_2)}$-error for Image reconstruction, Image inpainting, Image scaling and Image denoising  }
						\label{errimg}
						\vspace{0.7em}
						\rowcolors{2}{gray!10}{white}
						\resizebox{0.99\textwidth}{!}{%
							\begin{tabular}{c|cc|cc|cc|cc|}
								\toprule
								\rowcolor{gray!25}
								\multirow{2}{*}{$n$} 
								& \multicolumn{2}{c|}{Image reconstruction Fig \ref{g5}}
								& \multicolumn{2}{c|}{Image inpainting Fig \ref{g6}} 
								& \multicolumn{2}{c|}{Image scaling Fig \ref{g7}} 
								& \multicolumn{2}{c}{Image denoising Fig \ref{g8}} 
								\\
								\cmidrule(lr){2-3} \cmidrule(lr){4-5} \cmidrule(lr){6-7} \cmidrule(lr){8-9}
								\rowcolor{gray!25}
								& \makecell{$L^{(2,2)}$-error} 
								& \makecell{$L^{(2,4)}$-error}
								& \makecell{$L^{(3,3)}$-error} 
								& \makecell{$L^{(3,5)}$-error}
								& \makecell{$L^{(4,4)}$-error} 
								& \makecell{$L^{(4,5)}$-error}
								& \makecell{$L^{(5,5)}$-error} 
								& \makecell{$L^{(5,6)}$-error} 
								\\
								\midrule
								20  & 30.95 & 9.12 & 3.92 & 1.91 &3.11 & 2.41 &1.83 & 1.58 \\
								40  & 25.09 & 7.42   & 3.79 & 1.85 & 2.49 & 1.93 & 1.69 & 1.46 \\
								60  & 19.25 & 5.74 & 3.70 & 1.79 & 2.28 & 1.76 & 1.57 & 1.36 \\
								80  & 16.03  & 4.78   & 3.64  & 1.76 & 2.20 & 1.70 & 1.47& 1.28  \\
								100  & 13.82  & 4.14   & 3.63  & 1.75 & 2.10 & 1.62 & 1.38 & 1.20 \\
								\bottomrule
							\end{tabular}
						}
					\end{table}
				
					In Table~\ref{errimg}, we provide error estimates for the \(L^{(p_1,p_1)}\)-error and the mixed \(L^{(p_1,p_2)}\)-error for randomly chosen values of \(p_1\) and \(p_2\). From these results, we observe that the mixed norm yields a more detailed and directionally sensitive evaluation of the reconstruction error.
						This results indicates that the mixed norm offers a more refined and direction-sensitive evaluation of the reconstruction error.
					
						%
						%
						%


					\section{Conclusion}
					Based on the insights from Section \ref{1}, the theory of neural networks has been extensively explored due to its broad significance in both theoretical and applied domains.
					In this paper, we examine the approximation abilities of a family of \textit{multivariate Kantorovich-type neural network operators} as defined in (\ref{multi}). We study these operators within the framework of mixed norm Lebesgue spaces and mixed norm Orlicz spaces. 
					From Section \ref{2}, it is evident that $L^{\phi}$ 
					generates various function spaces depending on the choice of
					$\phi-$functions, including classical Lebesgue spaces, Logarithmic spaces and Exponential spaces. Hence, this paper provides a general framework for analyzing these operators across various function spaces simultaneously, providing a comprehensive understanding of their approximation behavior in different spaces with varying norm structures.
					Sections \ref{3} and \ref{4} are devoted to present boundedness and convergence results for $(K_{n} f)$ in $L^{\overrightarrow{\mathscr{P}}}(\mathcal{J})$ and $L^{\overrightarrow{\Phi}}(\mathcal{J})$ spaces respectively. 
					To demonstrate the convergence, in Section \ref{5}, we present a few examples with graphical representations (see Figures \ref{g1}-\ref{g4}) and estimates of approximation-error (see Tables \ref{t1}-\ref{or2}) in various function spaces utilizing specific sigmoidal functions, namely the logistic function and the hyperbolic tangent function.
					In addition we discuss its applications in image processing such as Image reconstruction (Fig \ref{g5}), Image inpainting  (Fig \ref{g6}), Image scaling  (Fig \ref{g7}) and Image denoising  (Fig \ref{g8}). From these figures, we demonstrate the capabilities of the operator, as evidenced by consistently higher SSIM and PSNR values with increasing $n$ (see Tables \ref{t11}-\ref{t33}). 	Subsequently, we provide  Table \ref{errimg} for error-estimation with respect to  $L^p$-error and mixed $L^{(p_1,p_2)}$-error.
						We further demonstrate the utility of the mixed norm structure in signal and image analysis affected by Gaussian and salt-and-pepper noise (see Fig~\ref{m11}-\ref{m22}), and provide an error-estimation  comparing the $L^p$ and mixed $L^{(p_1,p_2)}$	norms (Table \ref{m1}-\ref{m2}).


					\section{Information.} 
					
					\subsection{Ethic.}
					\noindent 
				The authors declare that this is an original work, not previously published or submitted for publication elsewhere. 
					
					\subsection{Funding.} 
					\noindent The research of the first author is funded by the University Grants Commission (UGC), New Delhi, India, through NTA Reference No : 231610034215.
					
					\subsection{Acknowledgement}
					\noindent 
					Priyanka Majethiya and Shivam Bajpeyi gratefully acknowledge SVNIT Surat, India, for the facilities and support provided during the course of this research.

					\subsection{Data availability.} 
					\noindent 
					This article does not involve any data; therefore, data sharing is not applicable.
					
					\subsection{Authorship contribution.}
					\noindent
					\textbf{Priyanka Majethiya:} Writing – Original Draft, Writing – Review and Editing, Conceptualization, Formal Analysis, Methodology, Mathematical Proofs, Visualization. 
					\textbf{Shivam Bajpeyi:} Writing – Original Draft, Writing – Review and Editing, Conceptualization, Formal Analysis, Methodology, Mathematical Proofs, Visualization, Supervision.

					\subsection{Conflict of interest.} 
					\noindent 
				The authors declare that there is no conflict of interest regarding the content of this article.

				\end{document}